\documentclass[final,a4paper,12pt]{report}
\usepackage[T1]{fontenc}
\usepackage[utf8]{inputenc}
\usepackage{amssymb}
\usepackage{amsfonts}
\usepackage{amsmath}
\usepackage{amsthm}
\usepackage{ulem}
\usepackage{palatino}
\usepackage{pstricks}
\usepackage{graphicx}
\usepackage{xcolor}
\usepackage{tikz}
\usepackage{float}
\usetikzlibrary{arrows,shapes}

\newtheoremstyle{s1}{9pt}{9pt}{\it}{\parindent}{\bf}{.}{0.5em}{}
\theoremstyle{s1}
\newtheorem{theorem}{Theorem}[chapter]
\newtheorem{theoremm}{Theorem}[chapter]

\newtheorem{conjecture}[theorem]{Conjecture}
\newtheorem{corollary}[theorem]{Corollary}

\newtheorem{definition}[theorem]{Definition}
\newtheorem{example}[theorem]{Example}

\newtheorem{lemma}[theorem]{Lemma}

\newtheorem{question}[theorem]{Question}
\newtheorem{proposition}[theorem]{Proposition}
\newtheorem{remark}[theorem]{Remark}

\newcommand{\field}[1]{\mathbb{#1}}
\newcommand{\C}{\field{C}}
\newcommand{\G}{\field{G}}
\newcommand{\N}{\field{N}}

\newcommand{\R}{\field{R}}

\newcommand{\K}{\field{K}}
\newcommand{\Le}{\field{L}}
\newcommand{\M}{\field{M}}
\newcommand{\s}{\field{S}}
\newcommand{\p}{\field{P}}

\DeclareMathOperator{\graph}{graph}
\DeclareMathOperator{\Fix}{Fix}

\begin{document}
\title{$\K$-uniruled sets in affine geometry}
\author{Michał Lasoń}

\begin{titlepage}
 \begin{center}
  \normalsize
Institute of Mathematics\\
Polish Academy of Sciences\\

\vfill
\hrule
\bigskip
{\sc \Large $\K$-uniruled sets in affine geometry}
\bigskip
\hrule
\bigskip

{\large \bf Michał Lasoń}

\vfill

{\large PhD thesis}

{\large Supervisor: Prof. dr hab. Zbigniew Jelonek}

\bigskip

\bigskip

{\normalsize Krak\'{o}w 2012}
 \end{center}
\end{titlepage}

\newpage
 \thispagestyle{empty}
 \mbox{}
 
 \tableofcontents
 
\pagebreak
\addcontentsline{toc}{chapter}{Streszczenie}
\chapter*{Streszczenie}

Celem niniejszej pracy doktorskiej jest badanie zbiorów $\K$-prostokreślnych w geometrii afinicznej. 

Pierwszy rozdział jest poświęcony własności $\K$-prostokreślności oraz wa\-run\-kom rów\-no\-waż\-nym.
Niech $\K$ będzie ciałem algebraicznie domkniętym. Krzywą $\Gamma$ nazywamy parametryczną jeżeli jest obrazem ciała $\K$ po\-przez odwzorowanie
regularne. Definicją \ref{k-uniruleddef} wprowadzamy stopień $\K$-pros\-to\-kreśl\-ności dla rozmaitości afinicznych. Rozmaitość afiniczna $X$ 
ma stopień $\K$-pros\-to\-kreśl\-ności co najwyżej $d$ jeżeli jest pokryta krzywymi pa\-ra\-me\-try\-cznymi stopnia co najwyżej $d$. 
W Propozycji \ref{k-uniruledofdeg} podajemy dwa warunki rów\-no\-waż\-ne. W pierwszym z nich żądamy tylko aby gęsty zbiór był pokryty 
takimi krzywymi. Kolejny warunek równoważny zapewnia istnienie dominującego, generycznie skończonego odwzorowania wielomianowego o stopniu ze względu 
na pierwszą współ\-rzęd\-ną co najwyżej $d$ z cylindra $\K\times W$ w rozmaitość $X$. Wprowadzamy nową Definicję \ref{k-uniruleddef} rozmaitości 
$\K$-prostokreślnych, która ma tą zaletę w stosunku do poprzedniej, że dzięki Propozycji \ref{indep} nie zależy od ciała podstawowego. W sekcji 
\ref{rr} przedstawiamy analogiczne  wyniki dla ciała liczb rzeczywistych.

Twierdzenie \ref{Sfkuni} mówi, że jeżeli $f:X\rightarrow Y$ jest generycznie skoń\-czo\-nym odwzorowaniem regularnym pomiędzy rozmaitościami 
afi\-nicz\-ny\-mi oraz rozmaitość $X$ jest $\K$-prostokreślna, to zbiór $\s_f$ punktów niewłaściwości odwzorowania $f$ jest również $\K$-prostokreślny. 
W drugim rozdziale zajmujemy się ograniczeniem od góry stopnia $\K$-prostokreślności zbioru $\s_f$ w zależności od stopnia odwzorowania $f$ oraz 
stopnia $\K$-pro\-sto\-kreśl\-no\-ści rozmaitości $X$. Dla ciała liczb zespolonych $\C$ przedstawiamy kompletne wyniki Twierdzenia \ref{cn} i \ref{multc}. 
Pierwsze z nich daje ograniczenie przez iloczyn stopnia odwzorowania $f$ oraz stopnia $\C$-prostokreślności rozmaitości $X$. Natomiast drugie mówi, 
że dla $f:\C^n\rightarrow\C^m$ stopnia $d$ stopień $\C$-prostokreślności $\s_f$ wynosi co najwyżej $d-1$. Przed\-sta\-wia\-my również wyniki dla dowolnego 
ciała algebraicznie domkniętego $\K$ - Twierdzenia \ref{kn}, \ref{kw}, \ref{kkw} oraz dla ciała liczb rzeczywistych $\R$ - Twier\-dze\-nia 
\ref{rn1}, \ref{rxw1}, \ref{multc1}.

W ostatnim rozdziale dowodzimy $\K$-prostokreślności zbiorów zwią\-za\-nych ze zbiorem punktów stałych działania grupy algebraicznej. 
Dokładnie w Twierdzeniu \ref{glowne} dowodzimy, że jeżeli nietrywialna, spójna grupa unipotentna działa efektywnie na rozmaitości afinicznej, 
to zbiór pun\-któw stałych tego działania jest $\K$-prostokreślny. W szczególności nie zawiera izolowanych punktów. 
Natomiast w Twierdzeniu \ref{glowne1} pokazujemy, że gdy nieskończona grupa algebraiczna działa efektywnie na przestrzeni afinicznej $\K^n$, to każda hiperpowierzchnia 
zawarta w zbiorze punktów stałych jest $\K$-prostokreślna.

\pagebreak
\addcontentsline{toc}{chapter}{Abstract}
\chapter*{Abstract}

The main goal of this thesis is to study $\K$-uniruled sets that appear in affine geometry. 

The first Chapter discusses the property of $\K$-uniruledness and its equi\-valent conditions.
Let $\K$ be an algebraically closed field. A curve $\Gamma$ is called parametric if it is 
an image of the field $\K$ under a regular map. In \ref{k-uniruleddef} we define the degree of $\K$-uniruledness 
for an affine variety $X$. The degree of $\K$-uniruledness 
of an affine variety $X$ is at most $d$ if $X$ is covered by parametric curves of degree at most $d$. 
In Proposition \ref{k-uniruledofdeg} we prove two equivalent conditions. In the first one we require that only a
dense subset of the variety $X$ is covered with such curves. Another equivalent condition asserts the existence of an 
affine variety $W$ with $\dim W=\dim X-1$, and a dominant polynomial map $\phi:\K\times W\rightarrow X$ of degree in the first coordinate
at most $d$. We state a new Definition \ref{k-uniruleddef} according to which $\K$-uniruled 
varieties are those with a finite degree of $\K$-uniruledness. One of the advantages of the new definition is that by 
Proposition \ref{indep} it does not depend on the base field. In section \ref{rr} we present analogous results for the field of real 
numbers.

Theorem \ref{Sfkuni} asserts that if $f:X\rightarrow Y$ is a generically finite regular mapping between affine 
varieties and $X$ is $\K$-uniruled, then the set $\s_f$ of points at which $f$ is not proper is also $\K$-uniruled. 
In the second Chapter we bound from above the degree of $\K$-uniruledness of the set $\s_f$ in terms of degree of 
a mapping $f$ and degree of $\K$-uniruledness of a variety $X$. For the field of complex numbers $\C$ we get 
the best possible results Theorems \ref{cn} and \ref{multc}. The first one gives a bound by the product of the degree 
of $f$ and degree $\C$-uniruledness of a variety $X$. The second theorem bounds by $d-1$ degree of $\C$-uniruledness 
of the set $\s_f$ for a mapping $f:\C^n\rightarrow\C^m$ of degree $d$. We also present results for an arbitrary 
algebraically closed field $\K$ - Theorems \ref{kn}, \ref{kw}, \ref{kkw} and for the field of real numbers $\R$ - Theorems \ref{rn1}, \ref{rxw1}, \ref{multc1}.

In the last Chapter we show that some sets associated with the set of fixed points of an algebraic group action are $\K$-uniruled. 
In Theorem \ref{glowne} we prove that if nontrivial connected unipotent group  acts effectively on an affine variety then 
the set of fixed points of this action is a $\K$-uniruled variety. In particular there are no isolated fixed points. 
In Theorem \ref{glowne1}, on the other hand, we show that if an arbitrary infinite algebraic group acts effectively on an affine space $\K^n$, 
then every hypersurface contained in the set of fixed points is $\K$-uniruled.

\pagebreak
\addcontentsline{toc}{chapter}{Acknowledgements}
\chapter*{Acknowledgements}

First and foremost I would like to thank my supervisor Zbigniew Jelonek for his constant and enormous help during my work on this thesis. Most of the
results presented here come from our two joint papers. It was a great pleasure and opportunity to learn from him. Throughout the thesis he provided
me with encouragement and useful ideas. I am sincerely grateful for his effort.

I would also like to express my gratitude to professors Rosa-Maria Miró-Roig, Jarosław Wiśniewski and Mikhail Zaidenberg for many interesting and 
fruitful discussions. I thank also Sławomir Cynk with whom I had written my master thesis, and who introduced me to algebraic geometry.

I acknowledge the administrative help of Anna Poczmańska.

Writing this thesis would not be possible without the constant support of my family. I thank a lot my mum for her love and motivation.

I thank my friends: Piotr Achinger, Agnieszka Bodzenta, Maria Donten-Bury, Michał Farnik, Wojciech Lubawski, Mateusz Michałek and Karol Palka 
for interesting discussions, questions, answers and a great working atmosphere.

I acknowledge a support of Polish Ministry of Science and Higher Education grant N N201 413139.
\pagebreak

\chapter{$\K$-uniruled varieties}

This Chapter can be treated as a preliminary to the rest of the thesis.

We are going to remind an old, well known definition of $\K$-uniruled varieties and introduce our new one. The relation between both definitions will 
be given. Additionally for $\K$-uniruled varieties we introduce a degree of $\K$-uniruledness.  

We end the Chapter by stating a few open problems concerning bounding the degree of $\K$-uniruledness. 

In the whole Chapter, unless stated otherwise, $\K$ is assumed to be an arbitrary algebraically closed field.

\section{Introduction}

We begin with a presentation of some necessary preliminaries.

\begin{proposition}\label{birrat} Let $\Gamma \subset \K^m$ be an affine curve. The following conditions are equivalent:
\begin{enumerate}
\item there exists a regular dominant map $\varphi:\K \rightarrow \Gamma$,
\item there exists a regular birational map $\varphi: \K \rightarrow \Gamma$.
\end{enumerate}
\end{proposition}

\begin{proof}
It is an immediate consequence of L\"{u}roth's Theorem \ref{luroth}. 
\end{proof}

\begin{theorem}[L\"{u}roth]\label{luroth}
Suppose $\Le$ and $\M$ are arbitrary fields (not necessarily algebraically closed), and $\xi$ is transcendental over $\Le$, such that 
$$\Le\subsetneq\M\subset\Le(\xi).$$ Then $\M=\Le(\eta)$ for some $\eta$ in $\Le(\xi)$.
\end{theorem}

\begin{proof}(see \cite{pi})
Suppose $\eta\in\Le(\xi)\setminus\Le$, then 
$$\eta=\frac{f(\xi)}{g(\xi)}$$
for some polynomials $f,g\in\Le[X]$, at least one of which is of positive degree. Moreover we may assume that $f$ and $g$ have no common factor in $\Le[X]$ of positive degree. Since $\eta\notin\Le$, we have
$$\deg(f(X)-g(X)\eta)=\max(\deg(f),\deg(g)).$$
Clearly $f(\xi)-g(\xi)\eta=0$, so $\xi$ is algebraic over $\Le(\eta)$, hence $\eta$ is transcendental over $\Le$. As a consequence $f(X)-g(X)\eta$ is irreducible in $\Le[X]$, and
$$\max(\deg(f),\deg(g))=[\Le(\xi):\Le(\eta)].$$

Choosing $\eta$ from $\M\setminus\Le$ shows that $\Le(\xi)$ is algebraic of finite degree over $\M$. Denote $[\Le(\xi):\M]=n$, then the minimal polynomial of $\xi$ over $\M$ can be written as
$$\frac{p(\xi)}{q(\xi)}\sum_{i=0}^{n}f_i(\xi)\xi^i=0,$$
where $$p,q,f_i\in\Le[X],\text{ }\frac{p(\xi)}{q(\xi)}f_i(\xi)\in\M,$$
and polynomials $f_i$ with no common factor of positive degree. We have $f_n(\xi)\neq 0$ and since $\xi$ is not algebraic over $\Le$
$$\eta:=\frac{f_j(\xi)}{f_n(\xi)}\notin\Le$$
for some $j$. Clearly $\eta\in\M$. Assume equation 
$$\eta=\frac{f(\xi)}{g(\xi)}$$
holds as in the first part of the proof. Since $\xi$ is a root of $f(X)-g(X)\eta\in\M[X]$, this polynomial is divisible in $\M[X]$ by $\frac{p(\xi)}{q(\xi)}\sum_{i=0}^{n}f_i(\xi)X^i$. Therefore 
$$f(X)g(\xi)-g(X)f(\xi)=r(X,\xi)\sum_{i=0}^{n}f_i(\xi)X^i,\;\;\;\;\;\;\;\;(\star)$$
for some polynomial $r\in\Le[X,Y]$. 

Say the degree  of $f(X)g(\xi)-g(X)f(\xi)$ in $\xi$ is $m$. This is the same as the degree in $X$, so $n\leq m$. Additionally
$$m\leq\max(\deg(f),\deg(g))\leq\max(\deg(f_j),\deg(f_n))\leq\max(\deg(f_0),\dots).$$
By $(\star)$, there must be equalities, and $r$ must be constant in $\xi$. But now $f(X)$ and $g(X)$ are both divisible by $r(X)$, hence $r$ is also constant in $X$. We get that $n=m$, and $\M=\Le(\eta).$  
\end{proof}

\begin{definition}
An affine curve $\Gamma\subset\K^m$ is called a \textup{parametric curve}, if equivalent conditions from Proposition \ref{birrat} hold.
\end{definition}

We are ready to present the formerly known definition of $\K$-uniruled varieties, which was introduced in \cite{st}.

\begin{definition}
An affine variety $X\subset\K^m$ is called \textup{$\K$-uniruled}, if for every point $x\in X$ there exists a parametric curve $l_x\subset X$ passing through $x$. In other words $X$ is covered by parametric curves.
\end{definition}

A parametric curve is an image of the field, which is irreducible, so it is also irreducible. Hence each parametric curve on a variety is contained in some irreducible component. So a variety is $\K$-uniruled if and only if all its irreducible components are $\K$-uniruled.

The following characterization of components of $\K$-uniruled varieties is known \cite{st}.

\begin{proposition}\label{k-uniruled}
Let $\K$ be an uncountable field and let $X\subset \K^m$ be an irreducible affine variety. The following conditions  are equivalent:
\begin{enumerate}
\item $X$ is $\K$-uniruled,
\item there exists an open, non-empty subset $U\subset X$, such that for every point $x\in U$ there exists a parametric curve $l_x\subset X$ passing through $x$, 
\item there exists an affine variety $W$ with $\dim W = \dim X-1$, and a regular dominant map $\phi: \K\times W\rightarrow X$. 
\end{enumerate}
\end{proposition}

\begin{proof}
Implication $(1)\Rightarrow (2)$ is trivial. 

To prove implication $(2)\Rightarrow (3)$ let us define
$$S_d=\{\varphi:\K\rightarrow\K^m\text{ such that }\varphi(\K)\subset X\text{ and }\deg\varphi=d\}.$$ 
Each $\varphi=(\varphi_1,\dots,\varphi_m)$, where $\varphi_i(t)=\sum_{j=0}^{d}a_{j}^{i} t^{j}$, corresponds to a point 
$$(a^1_0,\dots,a^m_d)\in(\K^{d+1})^m\setminus(\K^d\times\{0\})^m.$$ 
Conditions $\varphi(\K)\subset X$ are polynomial equations, so $S_d$ is a quasiprojective variety. Consider the morphism: 
$$F_d:\K\times S_d\ni (t,\varphi)\rightarrow \varphi(t)\in X.$$ 
Let us denote the image by $X_d:=F_d(\K\times S_d)$. We know that $U\subset\bigcup_{d\in N}X_d$, since for every point $x\in U$ there is a parametric curve $l_x\subset X$ passing through $x$. Let $\overline{X_d}$ be the closure of $X_d$ in $X$. From Baire's Theorem \ref{baire} for the Zariski topology there exists $d$ such that $\overline{X_d}=X$. Now the map 
$$F_d:\K\times S_d:\rightarrow X$$
is dominant, so when we restrict it to some irreducible component $Y$ (suppose $Y\subset \K^M$)  of $S_d$ the map $\Phi=F_d\vert_{\K\times Y}$ is 
still dominant. Suppose $\dim X=n$ and $\dim(Y)=s$, on an open subset of $X$ fibers of the map $\Phi$ are of dimension $s+1-n$. Let $x$ be one of 
such points. From the construction of the set $S_d$ we know that the fiber $F=\Phi^{-1}(x)$ does not contain any line of the type $\K\times\{y\}$, 
so in particular the image $F'$ of the fiber $F$ under projection $\K\times Y\rightarrow Y$ has the same dimension. For general linear subspace 
$L\subset \K^M$ of dimension $M+n-s-1$ the dimension of $L\cap F'$ equals to $0$. Let us fix such $L$, and let $R$ be any irreducible component of 
$L\cap Y$ intersecting $F'$. Now mapping the
$$\Phi\vert_{\K\times R}:\K\times R\rightarrow X$$
confirms the assertion, since it has one fiber of dimension $0$ (at $x$), the dimension of $R$ is $n-1$ (at most $n-1$, because of the $0$ dimensional fiber, 
at least $n-1$ because of the small dimension of $L$), so as a consequence it is dominant.   

To prove implication $(3)\Rightarrow (1)$ we note that for 
$$\phi: \K\times W\ni (t,w)\to \phi(t,w)\in X$$ 
there exists $d\in\N$ such that $\deg_t \phi \leq d$. Next we use implication $(3)\Rightarrow (1)$ from Proposition \ref{k-uniruledofdeg}, which gives us condition $(1)$. 
\end{proof}

Baire's Theorem reads as follows.

\begin{theorem}[Baire]\label{baire}
Let $\K$ be an uncountable field, let $X\subset\K^m$ be an irreducible affine variety, and let $X_d$ for $d\in\N$ be its closed subsets. If $\bigcup_{d\in N}X_d$ contains a non-empty open subset $U$ of $X$, then $X_d=X$ for some $d$.  
\end{theorem}

\begin{proof}
We are going to prove it by induction on the dimension of $X$. If $\dim X=0$, then $X$ is just a point, and the assertion is clearly trivial. 

When $\dim X>0$, then there exists a regular function $f:\K^m\rightarrow \K$ non-constant on $X$ (one of coordinates is good). 

Assume that the assertion is false, which means that all $X_d$ are of a lower dimension than $X$, and consider sets 
$$X^c=X\cap\{f=c\}$$ 
for $c\in\K$, which are of pure codimension one. When $U\cap X^c\neq\emptyset$, then it means that one of irreducible 
components $R$ of $X^c$ satisfies the conditions of the theorem with sets $R\cap X_d$ for $d\in\N$. From the inductive 
assumption we know that $R$ equals to $R\cap X_d$ for some $d$, so to one of irreducible components of $X_d$ (they are of the same dimension). 
But there are only countably many irreducible components of $X_d$ for some $d\in\N$, and uncountably many $c\in\K$ for 
which $U\cap X^c\neq\emptyset$, since open set $U$ intersects almost all $X_c$ (all, except the finite number). The contradiction 
shows that the hypothesis was false.
\end{proof}

\section{The degree of $\K$-uniruledness}

For an uncountable field $\K$ there is a nice characterization of $\K$-uniruled varieties, namely Proposition \ref{k-uniruled}. However to work over an arbitrary algebraically closed field we need a refined version of the definition. We introduce it in our papers \cite{jela}, \cite{jela2}. It coincides with the older one for uncountable fields.

Moreover with the new definition of $\K$-uniruledness we are able to measure the degree of $\K$-uniruledness of $\K$-uniruled varieties.

\begin{definition}
An affine curve $\Gamma\subset \K^m$ is called a \textup{parametric curve of degree at most $d$}, if there exists a polynomial dominant map $f:\K\rightarrow \Gamma$ of degree at most $d$ (by degree of $f=(f_1,\dots,f_m)$ we mean $\deg f:=\max_i \deg  f_i$).
\end{definition}

Now we prove the following:

\begin{proposition}\label{k-uniruledofdeg}
Let $X\subset \K^m$ be an irreducible affine variety of dimension $n$, and let $d$ be an integer. The following conditions  are equivalent:
\begin{enumerate}
\item for every point $x\in X$ there exists a parametric curve $l_x\subset X$ of degree at most $d$ passing through $x$, 
\item there exists a dense in the Zariski topology subset $U\subset X$, such that for every point $x\in U$ there exists a parametric curve $l_x\subset X$ of degree at most $d$ passing through $x$, 
\item there exists an affine variety $W$ with $\dim W = \dim X-1$, and a dominant polynomial map $\phi: \K\times W\ni (t,w)\to \phi(t,w)\in X$, such that $\deg_t \phi \leq d$. 
\end{enumerate}
\end{proposition}

\begin{proof}
Implication $(1)\Rightarrow (2)$ is obvious. 

We prove $(2)\Rightarrow (1)$. Suppose that 
$$X=\{x\in \K^m:f_1(x)=0,\dots,f_r(x)=0\}.$$
For $a=(a_1,\dots,a_m)\in \K^m$ and $b=(b_{1,1}:\dots:b_{d,m})\in\p^M$, where $M=dm-1$, let $$\varphi_{a,b}:\K\ni t\rightarrow (a_1+b_{1,1}t+\dots+b_{1,d}^dt^d,\dots,a_m+b_{m,1}t+\dots+b_{m,d}^dt^d)\in\K^m$$
be a parametric curve of degree at most $d$. Consider a variety and a projection
$$\K^m\times\p^M\supset V=\{(a,b)\in \K^m\times\p^M:\forall_{t,i}\;f_i(\varphi_{a,b}(t))=0\}\ni(a,b)\rightarrow a\in \K^m.$$
The definition of the set $V$ says that parametric curves $\varphi_{a,b}$ are contained in $X$. Hence the image of the projection is contained in $X$ and contains dense set $U$, since through every point of $U$ passes a parametric curve of degree at most $d$. But since $\p^M$ is complete and $V$ is closed we have that the image is closed, and hence it is the whole set $X$. 

Let us prove $(2)\Rightarrow (3)$. For some affine chart $V_j=V\cap \{b_j=1\}$ the above map is dominant. We consider the following dominant mapping
$$\Phi: \K\times V_j\ni (t,\phi)\rightarrow \phi(t)\in X.$$
After replacing $V_j$ by one of its irreducible components $Y\subset \K^M$ the map remains dominant. Suppose $\dim X=n$ and $\dim(Y)=s$, on an open subset of $X$ fibers of the map $\Phi'=\Phi\vert_{k\times Y}$ are of dimension $s+1-n$. Let $x$ be one of such points. From the construction of the set $V$ we know that the fiber $F=\Phi'^{-1}(x)$ does not contain any line of the type $\K\times\{y\}$, so in particular the image $F'$ of the fiber $F$ under projection $\K\times Y\rightarrow Y$ has the same dimension. For general linear subspace $L\subset \K^M$ of dimension $M+n-s-1$ the dimension of $L\cap F'$ equals to $0$. Let us fix such $L$, and let $R$ be any irreducible component of $L\cap Y$ intersecting $F'$. Now mapping 
$$\Phi'\vert_{\K\times R}:\K\times R\rightarrow X$$ 
confirms the assertion, since it has one fiber of dimension $0$ (at $x$), dimension of $R$ is $n-1$ (at most $n-1$, because of the $0$ dimensional fiber, at least $n-1$ because of the small dimension of $L$), so as a consequence it is dominant.   

To prove the implication $(3)\Rightarrow (2)$ it is enough to notice that for each $w\in W$ the map 
$$\phi_w: \K\ni t\rightarrow\phi(t,w)\in X$$
is a parametric curve of degree at most $d$ or it is constant. Image of $\phi$ contains an open dense subset, so after excluding points with 
infinite preimages (closed set, at most of codimension one) we get an open set $U$ with required properties. 
\end{proof}

We are ready to define the degree of $\K$-uniruledness - a parameter which measures degree of parametric curves covering a variety.

\begin{definition}\label{k-uniruleddef}
We say that an affine variety $X$ has \textup{degree of $\K$-uni\-ruled\-ness at most $d$}, if all its irreducible components satisfy the above conditions. A \textup{degree of $\K$-uniruledness} is the minimum $d$ for which it has degree of $\K$-uni\-ruled\-ness at most $d$. An affine variety is called \textup{$\K$-uniruled}, if it is $\K$-uniruled of some degree.
\end{definition}

The condition (3) from Propositions \ref{k-uniruled} and \ref{k-uniruledofdeg} is clearly the same, so both definitions of 
$\K$-uniruled varieties coincide for an uncountable field $\K$. From now on we are going to use only the second - Definition \ref{k-uniruleddef}.

\begin{example}
Let $X\subset \K^n$ be a hypersurface of degree $d<n$. It is well known that $X$ is covered by affine lines, so its degree of $\K$-uniruledness is one.
\end{example}

\begin{proof}
Let $f\in\K[x_1,\dots,x_n]$ be an equation of $X$. We have to show that through every point of $X$ passes a line. Without the loss of generality 
it is enough to assume that $O=(0,\dots,0)\in X$ and show that through $O$ passes a line (it is because we can make a linear change of coordinates 
and the degree of $f$ stays unchanged). A parametrization of such line $L$ is 
$$\K\ni t\rightarrow (a_1t,\dots,a_nt)\in\K^n,$$ 
for $(a_1,\dots,a_n)\in\p^{n-1}$. Line $L$ is contained in variety $X$ if and only if $f(a_1t,\dots,a_nt)=0$ for all $t\in\K$. It happens if and only if $f_i(a_1,\dots,a_n)=0$ for $i=1,\dots,d$ (where $f_i$ is the part of $f$ of degree $i$). Now we use the fact that $d$ equations define a variety in $\p^{n-1}$ of dimension at least $n-1-d$, which for $d<n$ is greater or equal to zero. It means that there exists $(a_1,\dots,a_n)\in\p^{n-1}$ for which the corresponding line is contained in $X$.
\end{proof}

\begin{proposition}\label{indep} The degree of $\K$-uniruledness of a variety does not depend on the base field. Namely let $\K\subset\Le$ be a field extension, and let $f_1,\dots,f_r\in\K[x_1,\dots,x_m]$. Then the following conditions are equivalent:
\begin{enumerate}
\item $\K X:=\{x\in\K^m:f_1(x)=0,\dots,f_r(x)=0\}\subset\K^m$ has degree of $\K$-uniruledness at most $d$,
\item $\Le X:=\{x\in\Le^m:f_1(x)=0,\dots,f_r(x)=0\}\subset\Le^m$ has degree of $\Le$-uniruledness at most $d$
\end{enumerate}
\end{proposition}
\begin{proof}
Implication $(1)\Rightarrow (2)$. Observe that the set $\K X$ is dense in $\Le X$. 
Indeed let $g\in\Le[x_1,\dots,x_m]$ and $g\equiv 0$ on $\K X$. Let $\mathcal{B}$ be any basis of vector space $\Le$ over $\K$. Now $g$ has a decomposition
$$g=b_1g_1+\dots+b_kg_k\text{ for }g_1,\dots,g_k\in\K[x]\text{ and distinct }b_1,\dots,b_k\in\mathcal{B}.$$
Now $g\equiv 0$ on $\K X$, so also for each $i$ there is $g_i\equiv 0$ on $\K X$. Hence for each $i$ $g_i$ belongs to the radical of the ideal $(f_1,\dots,f_r)$ in $\K[x]$. It implies that for each $i$ polynomial $g_i$ belongs to the radical of the ideal $(f_1,\dots,f_r)$ in $\Le[x]$, and as a consequence $g$ also belongs. So $g\equiv 0$ on $\Le X$, and it follows that $\K X$ is dense in $\Le X$. Of course any $\K$ parametric curve in $\K X$ extends to a $\Le$ parametric curve in $\Le X$. From the assumption and Proposition \ref{k-uniruledofdeg} condition $(2)$ we get that $X\subset\Le^m$ has degree of $\Le$-uniruledness at most $d$.

Implication $(2)\Rightarrow (1)$. As in the proof of Proposition \ref{k-uniruledofdeg} for 
$a=(a_1,\dots,a_m)\in \Le^m$ and $b=(b_{1,1}:\dots:b_{d,m})\in\p\Le^M$, where $M=dm-1$, let 
$$\varphi_{a,b}:\Le\ni t\rightarrow (a_1+b_{1,1}t+\dots+b_{1,d}^dt^d,\dots,a_m+b_{m,1}t+\dots+b_{m,d}^dt^d)\in\Le^m$$
be a parametric curve of degree at most $d$. We consider a variety and a projection
$$\Le^m\times\p\Le^M\supset V=\{(a,b)\in\Le^m\times\p\Le^M:\forall_{t,i}\;f_i(\varphi_{a,b}(t))=0\}\ni(a,b)\xrightarrow{p} a\in \Le^m.$$
From our assumption the image of the projection is the whole $\Le X$. But the above variety $V$ and also projection $p$ are defined over the 
smaller field $\K$. The anologous argument to the one from opposite implication shows that over $\K$ the image of the projection is also the 
whole variety $\K X$, which gives condition $(1)$. Indeed if for some $g\in\K[x]$ $g\equiv 0$ on $p(\K V)$ then $g\circ p$ belongs to the radical 
of the ideal $(h_1,\dots,h_s)$ in $\K[x]$, where $h_i$ are equations defining $V$. So $g\circ p$ belongs to the radical of the ideal 
$(h_1,\dots,h_s)$ in $\Le[x]$ and $g\equiv 0$ on $p(\Le V)$. It follows that $g$ belongs to the radical of the ideal $(f_1,\dots,f_r)$ in $\Le[x]$ 
which implies that $g$ belongs to the radical of the ideal $(f_1,\dots,f_r)$ in $\K[x]$.
\end{proof}

\section{A real field case}\label{rr}

For the whole section we assume that the base field is $\R$. 

Let us specify that by a \textit{real parametric curve of degree at most $d$} in a semialgebraic set $X\subset \R^n$ we mean the image of a real polynomial mapping $f:\R\to X$ of degree at most $d$. Thus in general a real parametric curve does not have to be algebraic, but only semialgebraic. The real counterpart of Proposition \ref{k-uniruledofdeg} is:

\begin{proposition}\label{real}
Let $X\subset \R^m$ be a closed irreducible semialgebraic set of dimension at least two, and let $d$ be an integer. The following conditions are equivalent:
\begin{enumerate}
\item for every point $x\in X$ there exists a parametric curve $l_x\subset X$ of degree at most $d$ passing through $x$,
\item there exists a dense in the classical topology subset $U\subset X$, such that for every point $x\in U$ there exists a parametric curve $l_x\subset X$ of degree at most $d$ passing through $x$,
\item for every polynomial mapping $f:X\to\R^n$, and every sequence $(x_k)_{k\in\N}\subset X$ such that $f(x_k)\to a\in \R^n$ there exists a semialgebraic curve $W$ and a generically finite polynomial map $\phi:\R\times W\ni (t,w)\to \phi(t,w)\in X$ such that $\deg_t \phi \leq d$, 
and there exists a sequence $(y_k)_{k\in\N}\in\R\times W$ such that $f(\phi(y_k))\to a$. Moreover, if $x_k\to\infty$ then also $\phi(y_k)\to\infty$.
\end{enumerate}
\end{proposition}

\begin{proof}
First we prove implication $(2)\Rightarrow (1)$. Suppose that 
$$X=\{x\in\R^m:f_1(x)=0,\dots,f_r(x)=0\;g_1(x)\geq 0,\dots, g_s(x)\geq 0\}.$$ 
For $a=(a_1,\dots,a_m)\in\R^m$ and $b=(b_{1,1},\dots,b_{d,m})\in\R^M$, where $M=dm$, let $$\varphi_{a,b}:\R\ni t\rightarrow (a_1+b_{1,1}t+\dots+b_{1,d}^dt^d,\dots,a_m+b_{m,1}t+\dots+b_{m,d}^dt^d)\in\R^m$$ 
be a parametric curve of degree at most $d$. If there exists a parametric curve of degree at most $d$ passing through $a$, then after reparametrization we can assume that it is $\varphi_{a,b}$ with $\sum_{i,j} b_{i,j}^2=1$. This means that $b\in S_{M-1}$, where $S_{M-1}$ denotes the unit sphere in $\R^M$ with the center in the origin. Consider a semialgebraic set
$$V=\{(a,b)\in\R^m\times S_{M-1}:\forall_{t,i}\;f_i(\varphi_{a,b}(t))=0, \forall_{t,j}\;g_j(\varphi_{a,b}(t))\geq 0\}.$$
The definition of the set $V$ says that parametric curves $\varphi_{a,b}(t)$ are contained in $X$. It is easy to see that $V$ is closed. For any $a\in X$, by the assumption there is a sequence of points $a_k\to a$, such that for every $k$ there is a parametric curve $\varphi_{a_k,b_k}\in V$. We can assume that $\Vert a_k\Vert<\Vert a\Vert+1$ for all $k$. The set $V$ is closed, the sequence $(a_k,b_k)_{k\in\N}\subset V$ is bounded, so there exists a subsequence $(a_{k_r},b_{k_r})_{r\in\N}$ which converges to $(a,b)\in V$. Now parametric curve $\varphi_{a,b}\subset X$ of degree at most $d$ passes through $a$.

To prove $(1)\Rightarrow (3)$ consider a semialgebraic set $V$ as above. We have a surjective mapping $$\Phi: \R\times V\ni (t,\varphi_{a,b})\to \varphi_{a,b}(t)\in X.$$ Let $f:X\to\R^n$ be a polynomial mapping,  
and let $f(x_k)\to a\in \R^n$ for a sequence $(x_k)_{k\in\N}\subset X$. Put $g=f\circ\Phi$. Hence there exists a sequence $z_k\in \R\times V$ such that $g(z_k)\to a$. Due to the real Curve Selection Lemma analogous to Lemma \ref{curve} there is a semialgebraic curve $W_1\subset \R\times V$ such that $a\in \overline{g(W_1)}$. Take $W_2=p_2(W_1)$, where $p_2: \R\times V \to V$ is a projection. If $W_2$ is a curve then let $W:=W_2$, if it is a point we take as $W$ any
semialgebraic curve in $V$ containing a point $p_2(W_1)$. Now $(W,\Phi\vert_{\R\times W})$ is a good pair.

Finally to prove $(3)\Rightarrow (2)$ it is enough to take $f$ being an identity in the third condition.
\end{proof}

Note that the first two conditions are the same as for general algebraically closed field in Proposition \ref{k-uniruledofdeg}.

\begin{definition}\label{r-uniruleddef}
We say that a closed semialgebraic set $X$ has \textup{degree of $\R$-uni\-ruled\-ness at most $d$}, if all its irreducible components satisfy the above conditions. A \textup{degree of $\R$-uniruledness} is the minimum $d$ for which it has degree of $\R$-uni\-ruled\-ness at most $d$. A closed semialgebraic set is called \textup{$\R$-uniruled}, if it is $\R$-uniruled of some degree.
\end{definition}

\begin{example}
Let $X=\{ (x,y)\in\R^2 : x\geq 0, y\ge 0\}$. It is easy to check that the  degree of $\R$-uniruledness of $X$ is two. It has a ruling with curves $l_a$ of the form $l_a=\{(a, t^2):t\ge 0\}$ for  $a\ge 0$.
\end{example}

\section{Remarks}

The statement of Proposition \ref{k-uniruled} suggests a question.

\begin{question}
Does equivalence from Proposition \ref{k-uniruled} hold also for countable fields?
\end{question}

However it is known that there exist K3 surfaces, which are not $\C$-uniruled, but their dense subset is covered by parametric curves. In this sense 
conditions $(1)$ and $(2)$ from  Proposition \ref{k-uniruledofdeg} with no upper bound on the degree of parametric curves are not equivalent.

\pagebreak

\chapter[Bounding degree of $\K$-uniruledness]{Bounding degree of $\K$-uniruledness of the non-properness set of a polynomial mapping}

In this Chapter we are going to recall the definition, motivation and some known facts about the set $\s_f$ of points at which regular mapping $f$ is not proper. It is known that if $f:X\rightarrow Y$ is a generically finite regular dominant map between affine varieties then the set $\s_f$ is a hypersurface in $Y$ or it is empty. If $X$ is additionally $\K$-uniruled then the set $\s_f$ is also $\K$-uniruled. The main goal of this Chapter is to bound from above the degree of $\K$-uniruledness of the set $\s_f$ in terms of the degree of the map $f$ and the degree of $\K$-uniruledness of $X$. We prove:

\begin{theoremm}
Let $\K=\C$, or $\R$, and let $f:\K^n\rightarrow \K^m$ be a generically
finite polynomial mapping of degree $d$. Then the set $\s_f$ has degree of $\K$-uniruledness at most $d-1$ or it is empty.
\end{theoremm}

\begin{theoremm}
Let $X$ be an affine variety with degree of $\C$-uniruledness at most $d_1$ and let
$f:X\rightarrow\C^m$ be a generically finite polynomial mapping of degree
$d_2$. Then the set $\s_f$ has degree of $\C$-uniruledness at most $d_1d_2$ or it is empty.
\end{theoremm}

\begin{theoremm}
Let $X$ be a closed semialgebraic set with degree of $\R$-uni\-ruled\-ness at most $d_1$, and let 
$f:X\rightarrow\R^m$ be a generically finite polynomial mapping of degree $d_2$. 
Then the set $\s_f$ has degree of $\R$-uniruledness at most $2d_1d_2$ or it is empty.
\end{theoremm}

\begin{theoremm}
Let $\K$ be an arbitrary algebraically closed field, and let
$f:\K^n\rightarrow\K^m$ be a generically finite polynomial mapping of degree
$d$. Then the set $\s_f$ has degree
of $\K$-uniruledness at most $d$ or it is empty.
\end{theoremm}

\section{The set $\s_f$ of points at which regular map $f$ is not proper}

We begin with some well-known facts concerning proper maps between topological spaces. It will provide a motivation for the definition of proper maps between affine varieties over arbitrary field.

\begin{definition}
Let $X,Y$ be topological spaces, and let $f:X\rightarrow Y$ be a continuous map. We say that $f$ is \textup{proper} if for every compact set $K\subset Y$ the set $f^{-1}(K)$ is also compact. 
\end{definition}

\begin{definition}
Let $X,Y$ be topological spaces, and let $f:X\rightarrow Y$ be a continuous map. We say that $f$ is \textup{proper at a point $y\in\overline{f(X)}$} if there exists an open neighborhood $U$ of $y$, such that $f\vert_{f^{-1}(U)}:f^{-1}(U)\rightarrow U$ is a proper map. 
\end{definition}

\begin{proposition}\label{topeq}
Let $X$ be a topological space, $Y$ be a locally compact topological space, and let $f:X\rightarrow Y$ be a continuous map. Then the following conditions are equivalent:
\begin{enumerate}
\item $f$ is proper,
\item $f$ is proper at every point of $\overline{f(X)}$.
\end{enumerate}
\end{proposition}

\begin{proof}
One implication is obvious. Namely if $f$ is proper then it is also proper at every point, we can take an open neighborhood $U=Y$. 

To prove the other implication suppose that $f$ is proper at every point. Let $K\subset Y$ be a compact set, so $K':=K\cap\overline{f(X)}$ is also 
compact. For any $y\in K'$ there exists an open neighborhood $U_y$ such that $f\vert_{f^{-1}(U_y)}$ is proper. From local compactness we have that 
for any $y\in K'$ there exists an open neighborhood $V_y$ such that $\overline{V_y}$ is compact subset of $U_y$. The family of sets $\{V_y\}_{y\in K'}$ forms an open cover of compact set $K'$. Hence there are $y_1,\dots,y_k\in K'$ such that $K'\subset V_{y_1}\cup\dots\cup V_{y_k}$. Therefore:
$$f^{-1}(K)=f^{-1}(K')=(f\vert_{f^{-1}(U_1)})^{-1}(K'\cap\overline{V_1})\cup\dots\cup (f\vert_{f^{-1}(U_k)})^{-1}(K'\cap\overline{V_k}).$$ 
The right hand side is a finite sum of compact sets, so $f^{-1}(K)$ is also compact.
\end{proof}

If $X,Y$ are affine varieties over $\C$, then there exists a purely algebraic condition for a regular map $f:X\rightarrow Y$ to be proper. 

\begin{proposition}
Let $X,Y$ be affine varieties over $\C$, and let $f:X\rightarrow Y$ be a regular map. Then $f$ is proper in the classical topology if and only if $f^*:\C[Y]\ni h\rightarrow h\circ f\in\C[X]$ is an integral extension of rings.
\end{proposition}


It an easy consequence of \cite{twwi} and the fact that monic polynomials with bound\-ed coefficients have bounded roots.

Motivated by the above equivalence we are going to extend the definition of proper maps to maps between affine varieties over 
arbitrary algebraically closed field. 

For maps between affine varieties the properties of being finite and of being proper coincide, hence we define finite maps. 

\begin{definition}\label{properdef}
Let $X,Y$ be affine varieties, and let $f:X\rightarrow Y$ be a regular, dominant map. We say that $f$ is \textup{finite} if the 
induced ring extension $f^*:\K[Y]\ni h\rightarrow h\circ f\in\K[X]$ is an integral extension of rings.
\end{definition}

Finally we are ready to introduce the set $\s_f$. Let $f:X\rightarrow Y$ be a generically finite regular map of affine varieties. A measure of non-properness of $f$ is the set $\s_f$ of points, at which the map $f$ is not proper. 

\begin{definition}
Let $X,Y$ be affine varieties, and let $f:X\rightarrow Y$ be a regular map. We say that $f$ is \textup{proper} at a point $y\in\overline{f(X)}$, 
if there exists a Zariski open neighborhood $U$ of $y$ such that $f^{-1}(U)$ is also affine, and 
$$f\vert_{f^{-1}(U)}:f^{-1}(U)\rightarrow U$$ is a finite map. We denote the set of points from $\overline{f(X)}$ at which $f$ is not proper by $\s_f$.
\end{definition}

The set $\s_f$ was first introduced by Jelonek in \cite{je93}, and it is the main object of study of the papers \cite{je99, je01, st, st2}. The set $\s_f$ indicates how far map $f$ is from a proper mapping, and therefore it is important in affine algebraic geometry as well as in applied mathematics (\cite{ha, je03, je05, sa}).

Being finite is a local property. We have the following equivalences in the spirit of Proposition \ref{topeq}. 

\begin{proposition}[\cite{sh}, page $62$, Theorem $5$] 
Let $X,Y$ be affine varieties, and let $f:X\rightarrow Y$ be a regular map. The following conditions are equivalent:
\begin{enumerate}
\item $f$ is finite,
\item for every open set $U\subset Y$ such that $f^{-1}(U)$ is affine the restriction $f\vert_{f^{-1}(U)}:f^{-1}(U)\rightarrow U$ is finite,
\item for every $y\in\overline{f(X)}$ the map $f$ is proper at $y$.
\end{enumerate}
\end{proposition}

The above Proposition enables us to extend Definition \ref{properdef} of finite maps.

\begin{definition}
Let $X,Y$ be quasiprojective varieties, and let $f:X\rightarrow Y$ be a regular map. We say that $f$ is \textup{finite} if for every point $y\in Y$ exists an affine neighborhood $U$ such that $f^{-1}(U)$ is affine and 
$$f\vert_{f^{-1}(U)}: f^{-1}(U)\rightarrow U$$
is a finite map between affine varieties.
\end{definition}

The following properties of the set of points, at which a map is not proper, will be frequently used.

\begin{proposition}\label{zlozenie}
Let $X,Y,Z$ be affine varieties, and let $f:X\rightarrow Y$, $g:Y\rightarrow Z$ be regular dominant maps. Then $\s_g\subset \s_{g\circ f}$.
\end{proposition}

\begin{proof}
The assertion is equivalent to the fact that complements satisfy $Z\setminus\s_{g\circ f}\subset Z\setminus \s_g$. Let $z\in Z\setminus \s_{g\circ f}$, then there exists an open set $z\in U$ such that the composition
$$(g\circ f)^{-1}(U)\xrightarrow{f}g^{-1}(U)\xrightarrow{g}U$$
is finite. So by the definition the composition
$$\K[U]\xrightarrow{g^*}\K[g^{-1}(U)]\xrightarrow{f^*}\K[(g\circ f)^{-1}(U)]$$
is an integral extension of rings. Of course then 
$$\K[U]\xrightarrow{g^*}\K[g^{-1}(U)]$$
is also an integral extension of rings. Hence the map
$$g^{-1}(U)\xrightarrow{g}U$$
is finite and $z\in Z\setminus \s_g$. 
\end{proof}

\begin{corollary}\label{isoo}
If $f:X\rightarrow Y$ is an isomorphism between affine varieties, then $f$ is finite.
\end{corollary}

\begin{proof}
There exists $g:Y\rightarrow X$ such that $f\circ g=id_Y$. Then by Proposition \ref{zlozenie} $\s_{f}\subset \s_{id_Y}=\emptyset$, so $f$ is finite.
\end{proof}

\begin{proposition}\label{zlozeniesk}
Let $X,Y,Z$ be affine varieties, $g:Y\rightarrow Z$ be a regular dominant map, and let $f:X\rightarrow Y$ be a finite map. Then $\s_g=\s_{g\circ f}$.
\end{proposition}

\begin{proof}
Inclusion $\s_g\subset\s_{g\circ f}$ follows from Proposition \ref{zlozenie}.

To prove the other inclusion let $z\in Z\setminus\s_g$, then there exists an open set $z\in U$ such that 
$$g^{-1}(U)\xrightarrow{g}U$$
is finite. So by the definition 
$$\K[U]\xrightarrow{g^*}\K[g^{-1}(U)]$$
is an integral extension of rings. From the assumption we know that $f$ is finite, so
$$\K[g^{-1}(U)]\xrightarrow{f^*}\K[(g\circ f)^{-1}(U)]$$
is also an integral extension of rings. Now the composition 
$$\K[U]\xrightarrow{g^*}\K[g^{-1}(U)]\xrightarrow{f^*}\K[(g\circ f)^{-1}(U)]$$
is also an integral extension of rings. Hence the map
$$(g\circ f)^{-1}(U)\xrightarrow{g\circ f}U$$
is finite and $z\in Z\setminus\s_{g\circ f}$. 
\end{proof}

\begin{proposition}[\cite{je00}]\label{graph}
Let $f:X\rightarrow Y$ be a generically finite mapping between affine varieties $X\subset\K^n$ and $Y\subset\K^m$. Let 
$$\graph(f)=\{(x,y)\in X\times Y:y_i=f_i(x)\}$$ 
and let $\overline{\graph(f)}$ be its closure in $\p^n\times\K^m$. Then there is an equality $$\s_f=p_2(\overline{\graph(f)}\setminus\graph(f)),$$ 
where $p_2$ denotes the projection $\p^n\times\K^m\rightarrow\K^m$.
\end{proposition}

\begin{proof}
Observe that morphism 
$$i:X\ni x\rightarrow (x,f(x))\in\graph(f)$$
is finite and that $f=p_2\circ i$. Now due to Proposition \ref{zlozeniesk} $\s_f=\s_{p_2}$. So it is enough to show that $\s_{p_2}=p_2(\overline{\graph(f)}\setminus\graph(f))$.

We will show the inclusion $\subset$. Suppose 
$$y\notin p_2(\overline{\graph(f)}\setminus\graph(f)).$$
Then there exists an open neighborhood $U$ of $y$, which is disjoint from $$p_2(\overline{\graph(f)}\setminus\graph(f)).$$ 
Since the projection $p_2$ is proper, the mapping
$$p_2\vert_{p_2^{-1}(U)}: p_2^{-1}(U)\rightarrow U$$ 
is also proper. The map $f\vert_{f^{-1}(U)}$ has finite fibers, hence also the map $p_2\vert_{p_2^{-1}(U)}$ has finite fibers. Due to \cite{ii} (page $142$, Theorem $2.27$) the map $p_2\vert_{p_2^{-1}(U)}$ is finite, and consequently $y\notin \s_{p_2}$. 

The opposite inclusion is also true, however since we will not use it, we do not prove it.
\end{proof}

The following property is well-known as curve selection lemma:

\begin{lemma}\label{curve}
Let $X,Y$ be affine varieties, let $f:X\rightarrow Y$ be a regular dominant map, and let $y\in \s_f$. Then there exists an affine curve $\Gamma\subset X$, such that $y\in \s_{f\vert_{\Gamma}}$.
\end{lemma}

\begin{proof}
It is enough to show that if $X,Y$ are affine varieties, $\dim X=n>1$, $f:X\rightarrow Y$ is a regular dominant map, and let $y\in \s_f$, then there exists an affine variety $X'\subset X$, such that $$\dim X'=\dim X-1\text{ and }y\in \s_{f\vert_{X'}}$$ Then the assertion follows from the induction.

Suppose $X\subset\K^m$. Without the loss of generality we can assume $X$ is irreducible. From Proposition \ref{graph} we know that 
$$y\in \s_f=p_2(\overline{\graph(f)}\setminus \graph(f)).$$ 
So there exists
$$x\in\overline{\graph(f)}\setminus\graph(f)=\overline{\graph(f)}\cap (\p^m\setminus\K^m),$$
such that $(x,y)\in\overline{\graph(f)}$.
Consider irreducible hypersurfaces $H$ in \newline $\overline{\graph(f)}$ passing through $(x,y)$ and different 
from $\overline{\graph(f)}\setminus\graph(f)$. Take any such $H$, now $X'$ equal to projection of $H\cap\K^m$ to $X$ satisfies our requirements.
\end{proof}

Two most important properties of the set $\s_f$ are the following:

\begin{theorem}[\cite{je99}, page $5$, Theorem $3.8$]\label{hiper}
Let $X,Y$ be affine varieties, and let $f:X\rightarrow Y$ be a regular generically finite map. Then the set $\s_f$ is a hypersurface in $\overline{f(X)}$ or it is empty. 
\end{theorem}


The second property is that if $X$ is additionally $\K$-uniruled, then the set $\s_f$ is also $\K$-uniruled. We are going to prove it. First we will consider the case of surfaces.

Let $X$ be a smooth projective surface and let $D=\sum_{i=1}^n D_i$ be a simple normal crossing divisor on $X$ (here we consider only reduced divisors). Let $\graph(D)$ be a graph of $D$, with vertices $D_i$, and an edge between $D_i$ and $D_j$ for each point of intersection of $D_i$ and $D_j$. 

\begin{definition}
We say that a simple normal crossing divisor $D$ on a smooth surface $X$ is a tree if $\graph(D)$ is a tree (it is connected and acyclic).
\end{definition}

The following fact is obvious.

\begin{proposition}\label{acykl}
Let $X$ be a smooth projective surface and let $D\subset X$ be a divisor which is a tree. If $D', D''\subset D$ are connected divisors without common components, then $D'$ and $D''$ have at most one common point.
\end{proposition}

\begin{theorem}\label{Sfsurface}
Let $\Gamma$ be an affine curve, and let $f:\Gamma\times\K\rightarrow\K^m$ be a generically finite mapping. Then the set $S_{f}$ is $\K$-uniruled or it is empty.
\end{theorem}

\begin{proof}
Take an affine normalization $\nu:\Gamma^{\nu}\rightarrow\Gamma$ (\cite{sh}) of the curve $\Gamma$. From Proposition \ref{zlozeniesk} we get that $\s_f=\s_{\nu\circ f}$, because normalization is a finite mapping. Hence we can assume that the curve $\Gamma$ is smooth. 

Let $\overline{\Gamma}$ be a smooth completion of $\Gamma$, and let us denote $\overline{\Gamma}\setminus\Gamma=\{a_1,\dots,a_l\}$. Let $X=\Gamma\times\K$ and $\overline{X}= \overline{\Gamma}\times\p^1$ be a projective closure of $X$. The divisor 
$$D=\overline{X}\setminus X=\overline{\Gamma}\times\infty+\sum_{i=1}^l \{a_i\}\times\p^1$$
is a tree. We can resolve points of indeterminacy of the rational map $f:\overline{X}\dashrightarrow\p^m$.
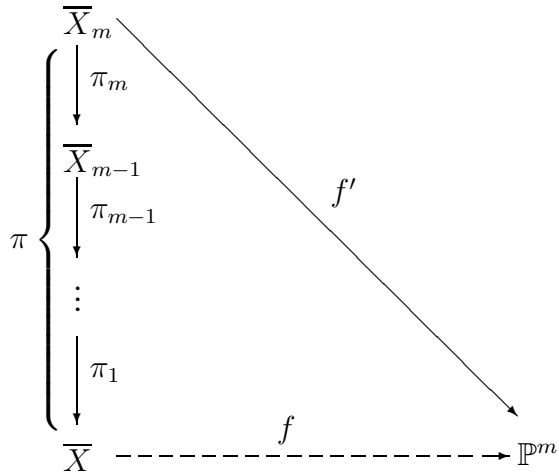
\vspace{5mm}

\begin{center}
\begin{picture}(240,160)(-40,40)
\put(-20,117.5){\makebox(0,0)[l]{$\pi\left\{\rule{0mm}{2.7cm}\right.$}}
\put(0,205){\makebox(0,0)[tl]{$\overline{X}_m$}}
\put(0,153){\makebox(0,0)[tl]{$\overline{X}_{m-1}$}}
\put(4,105){\makebox(0,0)[tl]{$\vdots$}}
\put(0,40){\makebox(0,0)[tl]{$\overline{X}$}}
\put(170,40){\makebox(0,0)[tl]{$\p^m$}}
\put(80,50){\makebox(0,0)[tl]{$f$}}
\put(100,140){\makebox(0,0)[tl]{$f'$}}
\put(10,70){\makebox(0,0)[tl]{$\pi_1$}}
\put(10,130){\makebox(0,0)[tl]{$\pi_{m-1}$}}
\put(10,180){\makebox(0,0)[tl]{$\pi_m$}}
\put(5,190){\vector(0,-1){30}} \put(5,140){\vector(0,-1){30}}
\put(5,80){\vector(0,-1){33}}
\multiput(20,35)(8,0){17}{\line(1,0){5}}
\put(157,35){\vector(1,0){10}} \put(20,200){\vector(1,-1){150}}
\end{picture}
\end{center}
\vspace{5mm}
 
Denote the proper transform of $\overline{\Gamma}\times\infty$ by $\pi^*(\overline{\Gamma}\times\infty)$. It is an easy observation that $f'(\pi^*(\overline{\Gamma}\times\infty))\subset H_\infty$, where $H_\infty$ is the hyperplane at infinity of $\p^m$. 

Note that the divisor $D'=\pi^*(D)$ is a tree, since each $\pi_i$ is a blow up. Moreover irreducible components of $D'$ are $\pi^*(\overline{\Gamma}\times\infty)$ and $\p^1$'s. The curve $C=f'^{-1}(H_\infty)\subset D'$ is a complement of a semi-affine variety $f'^{-1}(\K^m)$ hence it is connected (for details see \cite{je99}, Lemma~$4.5$). 

Now $\s_f=f'(D'\setminus C)$.
Let $R\subset \s_f$ be an irreducible component. From Proposition \ref{hiper} we know that $R$ is a curve. So there is an irreducible curve $Z\subset D'$, such that $R=f'(Z\setminus C)$. By Proposition \ref{acykl} we have that $Z$ has at most one common point with $C$. Of course it has at least one, since every morphism from irreducible projective variety to an affine variety is constant. So $R=f'(Z\setminus C)=f'(\K)$. This completes the proof due to Proposition \ref{k-uniruledofdeg} and the fact that $\s_f$ consists only of finite number of parametric curves.
\end{proof}

\begin{theorem}[\cite{je10}, page $3673$, Theorem $4.1$]\label{Sfkuni}
Let $X,Y$ be affine varieties, and let $f:X\rightarrow Y$ be a regular dominant map. If $X$ is additionally $\K$-uniruled, then the set $\s_f$ is also $\K$-uniruled or it is empty.
\end{theorem}

\begin{proof}
Due to Proposition \ref{k-uniruledofdeg} there exists a regular dominant map $g:\K\times W\rightarrow X$. From Proposition \ref{zlozenie} we have 
that $\s_f\subset \s_{g\circ f}$, and from Proposition \ref{hiper} both sets are of the same pure dimension. Therefore it is enough to prove the assertion for $f:\K\times W\rightarrow Y$. 

Let $y\in \s_f$. Using Lemma \ref{curve} we get that there is a curve $\Gamma\subset\K\times W$ such that $f\vert_{\Gamma}$ is not proper at $y$. Let $p_2:\K\times W\rightarrow W$ be a projection, and $\Gamma'=p_2(\Gamma)$. Of course $\Gamma'$ is still a curve, because otherwise $\Gamma=\K$ and $f\vert_{\Gamma}$ would be finite. We have that $\Gamma\subset\K\times\Gamma'$, so 
$$y\in \s_{f\vert_{\Gamma}}\subset \s_{f\vert_{\K\times\Gamma'}}\subset \s_f.$$ 
From Proposition \ref{Sfsurface} we get that there exists a parametric curve contained in $\s_f$ passing through $y$. By Propositions \ref{k-uniruled} and \ref{k-uniruledofdeg} this completes the proof for uncountable fields $\K$. For countable fields we extend $X,Y,f$ to an uncountable field $\K\subset\Le$ and use Proposition \ref{indep}.
\end{proof}

To motivate one conjecture we will need one more property of the set $\s_f$. It describes its degree as a hypersurface.

\begin{theorem}[\cite{je93}, page $264$, Theorem $15$]\label{hiperc}
Let $f:\C^n\rightarrow\C^n$ be a dominant polynomial  map. Then the degree of a hypersurface $\s_f$ is at most $$\frac{\prod_i\deg f_i-\mu(f)}{\min_i\deg f_i},$$ where $\mu(f)$ denotes the multiplicity of $f$.
\end{theorem}


\section{A complex field case}\label{c}

For the whole section we asume that the base field is $\C$. 

The condition that a map is finite at a point $y$ is equivalent to the fact that it is locally
proper in the classical topology sense (there exists an open neighborhood
$U$ of $y$ such that $f^{-1}(\overline{U})$ is compact). This
characterization gives the following:

\begin{proposition}[\cite{je93}, page $260$, Proposition $3$]\label{topo}
Let $f:X\rightarrow Y$ be a generically finite mapping between
affine varieties over $\C$. Then $y\in \s_f$ if and only if there exists a
sequence $(x_n)_{n\in\N}$, such that $\vert x_n\vert\rightarrow \infty$ and
$f(x_n)\rightarrow y$.
\end{proposition}

We will need a classical theorem of complex analysis.



\begin{theorem}[Rouch\'{e} \cite{ru}]\label{rouche}
Let $f,g:\C\rightarrow\C$ be holomorphic functions. Suppose that $\vert g(z)\vert<\vert f(z)\vert$ for every $z\in\partial D$, where $D$ is a region. Then $f$ and $f+g$ have the same number of zeros (counted with multiplicities) inside $D$. 
\end{theorem}

We are ready to prove the following:

\begin{theorem}\label{cn}
Let  $f:\C^n\rightarrow \C^m$ be a generically finite polynomial mapping of
degree $d$. Then the set $\s_f$ has degree of $\C$-uniruledness at most $d-1$ or it is empty.
\end{theorem}

\begin{proof}  Let $y\in \s_f$, by an affine transformation we can assume that 
$$y=O=(0,0,\dots,0)\in \C^m.$$
From the same reason we can assume that $O\notin f^{-1}(\s_f)$.

Due to Proposition \ref{topo} there exists a sequence of points
$\lvert x_k\rvert\to\infty$ such that $f(x_k)\to O$. Let us consider lines
$$L_k(t)=tO+(1-t)x_k=(1-t)x_k,\;\; t\in\C.$$
Denote $l_k(t)=f(L_k(t))$.
Infinite fibers cover only a subset of codimension at least $2$ in $\C^n$, and due to Proposition \ref{hiper} the codimension of $\s_f$ equals to $1$. Thus by Proposition \ref{k-uniruledofdeg} we can assume that $\deg l_k>0$. Note also that $\deg l_k\leq d$. Each curve $l_k$ is given by $m$ polynomials of one variable: 
$$l_k(t)=(\sum_{i=0}^d a^1_i(k)
t^i,\dots,\sum_{i=0}^d a^m_i(k) t^i).$$ 
Hence we can associate $l_k$ with a point
$$(a^1_0(k),\dots,a^1_d(k);\ldots;a^m_0(k),\dots,a^m_d(k))\in \C^N.$$

For each $i$ when $k\to \infty$, then $a^i_0(k)\to 0$. So we
can change the parametrization of $l_k$ by putting $t\rightarrow
\lambda_k t$, in such a way that $\Vert l_k\Vert=1,$ for $k\gg 0$
(we consider here $l_k$ as an element of $\C^N$ with the Euclidean
norm). 

Now, since unit sphere is compact, there exists a subsequence $(l_{k_r})_{r\in\N}$ of $(l_k)_{k\in\N}$, which is convergent to a polynomial mapping 
$$l : \C\rightarrow \C^m,$$ with $l (0) = O$. Moreover, $l$ is non-constant, because $\Vert l \Vert
= 1$, and $l(0)= O$. 

Choosing a subsequence we can also assume that the limit
$$\textit{lim}_{k\rightarrow \infty}\lambda_k=\lambda$$
exists in the compactification of the field $\C$. Consider two cases:
\begin{enumerate}
\item $\lambda\in\C$ is finite. If $k\to\infty$, then $L_k(\lambda_kt)=(1-\lambda_kt)x_k\to\infty$ for $t\not=\lambda^{-1}$.
\item $\lambda=\infty$. If $k\to\infty$, then $\Vert L_k(\lambda_k t)\Vert \geq (\vert\lambda_kt\vert-1)\Vert x_k\Vert$, and hence $\Vert L_k(\lambda_k t)\Vert \to
 \infty$ for every $t\neq 0$.
\end{enumerate}

On the other hand if $k\to\infty$
$$f(L_k(\lambda_k t))=l_k(\lambda_kt)\to l(t),$$ 
so using once more Proposition \ref{topo} this means that the curve $l$ is contained in the set $\s_f$. As a consequence the set $\s_f$ has 
degree of $\C$-uniruled\-ness at most $d$.

To complete the proof we should show that $\deg l<d$. It is enough to prove that $\deg l<\deg l_k$ for some $k$. Suppose the contrary, what is $\deg  l=\deg  l_k$ for all $k$.
Let 
$$l(t)=(l_1(t),\dots,l_m(t))\text{ and }l_k(t)=(l_1^k(t),\dots,l_m^k(t)).$$ We can assume that a component
$l_1(t)$ has the maximum degree. Denote $a=(a_1,\dots,a_m)=f(O)$. All roots of
the polynomial $l_1(t)-a_1$ are contained in the interior of some disc 
$$D=\{ t\in \C : \vert t\vert<R\}.$$ 
Let 
$$\epsilon =\inf\{ \vert l_1(t)-a_1\vert: t\in\partial D\}.$$ 
For $k\gg 0$ we have 
$$\vert (l_1-a_1)-(l_1^k-a_1)\vert_D<\epsilon.$$ 

Consequently by Rouch\'{e} Theorem \ref{rouche} these polynomials
have the same number of  zeros (counted with multiplicities) in $D$. In particular
zeros of $l_1^k-a_1$ are bounded. All curves $L_k$ pass through $O$, so all $l_k$ pass through $a=f(O)$. These means that there exists a sequence 
$t_k$ such that $l_k(t_k)=a$. We have just showed that $|t_k|<R$, since $t_k$ is a root of the polynomial $l_1^k-a_1$. 

So we can assume that the sequence $t_k$ converges to some $t_0$. When we pass to the limit we get $l(t_0)=a$, this is a contradiction, since $a=f(O)\not\in \s_f$.
\end{proof}

In a similar way we can prove the following generalization of
Theorem \ref{cn}:

\begin{theorem}\label{cxw}
Let  $X=\C\times W\subset \C\times \C^n$ be an affine cylinder and
let 
$$f:\C\times W\ni (t,w)\rightarrow (f_1(t,w),\dots,f_m(t,w))\in
\C^m$$ 
be a generically finite polynomial mapping. Assume that for every $i$ there is
$\deg _tf_i\leq d$. Then the set $\s_f$ has degree of $\C$-uniruledness at most $d$ or it is empty.
\end{theorem}

\begin{proof} As before for $y\in \s_f$ because of an affine transformation we can assume that 
$$y=O=(0,0,...,0)\in \C^m.$$ 

Due to Proposition \ref{topo} there exists a sequence of points $(a_k,w_k)_{k\in\N}\subset\C\times W$, such that $$\vert(a_k,w_k)\vert\to\infty\text{ and }f(a_k,w_k)\to y.$$ 
Let us consider lines 
$$L_k(t)=((1-t)a_k, w_k),\;\; t\in\C.$$ 
Denote $l_k(t)=f(L_k(t))$. 

Infinite fibers cover only a subset of codimension at least $2$ in $\overline{f(\C\times W)}$, while due to Proposition \ref{hiper} the codimension of $\s_f$ equals to $1$. Thus by Proposition \ref{k-uniruledofdeg} we can assume that $\deg l_k>0$. Note also that $\deg l_k\leq d$.  Each curve $l_k$ is given by $m$ polynomials of one
variable: 
$$l_k(t)=(\sum_{i=0}^d a^1_i(k) t^i,\dots,\sum_{i=0}^d
a^m_i(k) t^i).$$
As before we can associate $l_k$ with a point
$$(a^1_0(k),\dots,a^1_d(k);\ldots;a^m_0(k),\dots,a^m_d(k))\in \C^N.$$

For each $i$ when $k\to \infty$, then $a^i_0(k)\to 0$. Thus we
can change the parametrization of $l_k$ by $t\rightarrow
\lambda_k t$, in such a way that $\Vert l_k\Vert=1,$ for $k\gg 0$
(we consider $l_k$ as an element of  $\C^N$ with the Euclidean
norm). 

Compactness of unit sphere implies that there exists a
subsequence $(l_{k_r})_{r\in\N}$ of $(l_k)_{k\in\N}$, which is convergent to a
polynomial mapping $$l : \C\rightarrow \C^m,$$ with $l(0) = O$.
Moreover, $l$ is non-constant, because $\Vert l \Vert = 1$ and
$l(0)= O$. 

We can also assume that the limit
$$\textit{lim}_{k\rightarrow \infty}\lambda_k=\lambda$$ 
exists in the compactification of the field $\C$. 
We consider two cases:
\begin{enumerate}
 \item $\lambda$ is finite. If $k\to\infty$, then $L_k(\lambda_kt)=((1-\lambda_kt)a_k,w_k)\to\infty$ for $t\not=\lambda^{-1}$.
\item $\lambda=\infty$. If $k\to\infty$, then $\Vert L_k(\lambda_k t)\Vert \geq \max((\vert\lambda_kt\vert-1)\vert a_k\vert,\Vert w_k\Vert)$, and $\Vert L_k(\lambda_k t)\Vert \to \infty$ for every $t\neq 0$.
\end{enumerate}

But on the other hand 
$$f(L_k(\lambda_k t))=l_k(\lambda_kt)\to l(t),$$ so using once more Proposition \ref{topo} we get that the curve $l$ is contained in the set $\s_f$. Thus the set $\s_f$ has degree of $\C$-uniruledness at most $d$.
\end{proof}

\begin{corollary}\label{wn}
Let  $f=(f_1,...,f_m) :\C^n\rightarrow \C^m$ be a generically
finite polynomial mapping with $d=\min_j\max_i\deg _{x_j}f_i$. Then
the set $\s_f$ has degree of $\C$-uniruledness at most $d$ or it is empty.
\end{corollary}

\begin{theorem}\label{multc}
Let $X$ be an affine variety with degree of $\C$-uniruledness at
most $d_1$ and let $f:X\rightarrow \C^m$ be a
generically finite polynomial mapping of degree $d_2$. Then the set $\s_f$ has degree of $\C$-uniruledness at
most $d_1d_2$ or it is empty.
\end{theorem}

\begin{proof}
By Definition \ref{k-uniruleddef} there exists an affine variety
$W$ with $\dim W = \dim X-1$, and a dominant polynomial map 
$$\phi:\C\times W\rightarrow X$$
 of degree in the  first coordinate at
most $d_1$. Equality $\dim  \C\times W = \dim  X$ implies that
$\phi$ is generically finite, hence also generically finite is the map
$$f\circ\phi:\C\times W\rightarrow \C^m.$$ 
It is of degree in the  first coordinate at
most $d_1d_2$. Due to Theorem \ref{cxw} the set $\s_{f\circ\phi}$ 
has degree of $\C$-uniruledness at most $d_1d_2$.
By Proposition \ref{zlozenie} there is an inclusion $\s_f\subset\s_{f\circ\phi}$. From Theorem
\ref{hiper} we know that if not empty, then both sets are of pure
dimension $\dim X-1$, so components of $\s_f$ are components of
$\s_{f\circ\phi}$. This implies the assertion.
\end{proof}

\begin{example}
Consider the following polynomial mapping 
$$f: \C^n\ni (x_1,\ldots, x_n)\to (x_1, x_1x_2,\ldots,x_1x_n)\in \C^n.$$ 
We have $\deg f=2$ and $\s_f=\{x\in\C^n:x_1=0\}$. The set $\s_f$ has degree of $\C$-uniruledness equal to one. This shows that in general Theorems \ref{cn}, \ref{cxw} and Corollary \ref{wn} can not be improved.
\end{example}

\begin{example}
Consider the following polynomial mapping 
$$f: \C^2\ni (x,y)\to (x+(xy)^d, xy)\in \C^2.$$ 
We have $\deg f =2d$ and $\s_f=\{(s,t)\in\C^2:s=t^d\}$. The set $\s_f$ has degree of $\C$-uniruledness equal to $d$. This shows that in general Corollary
\ref{wn} can not be improved.
\end{example}

\begin{example}
For $n>2$ let $X=\{ x\in \C^n : x_1x_2=1\}$, and 
$$f:X\ni (x_1,\ldots, x_n)\to (x_2,\ldots, x_n)\in\C^{n-1}.$$ The variety $X$ has degree of $\C$-uniruledness equal to one, moreover we have $\deg f=1$ and $\s_f=\{x\in\C^{n-1}:x_1=0\}$. So the set $\s_f$ has degree of $\C$-uniruledness one. This shows that in general Theorems \ref{cxw}, \ref{multc} can not be improved.
\end{example}

\begin{remark}
Let us note that by Proposition \ref{indep} all results from this
section remain true for arbitrary algebraically closed field of
characteristic zero.
\end{remark}

\section{A real field case}\label{r}

For the whole section we assume that the base field is $\R$. 

Let $X\subset\R^m$ be a closed semialgebraic set, and let $f: X\to\R^n$ be a polynomial mapping. 
As in the complex case we say that $f$ is not proper at a point $y\in\R^n$ if there is no neighborhood $U$ of $y$ such that $f^{-1}(\overline{U})$ 
is compact. The set of all points $y\in\overline{f(X)}$ at which the mapping $f$ is not proper we denote as before by $\s_f$. This set is also closed and semialgebraic \cite{je02}. The results of \cite{je02} can be generalized as follows:

\begin{theorem}\label{rn1}
Let $f:\R^n\rightarrow\R^m$ be a generically finite polynomial mapping of degree $d$. Then the set $\s_f$ has degree of $\R$-uniruledness at most $d-1$ or it is empty.
\end{theorem}

\begin{theorem}\label{rxw1}
Let $X=\R\times W\subset \R\times \R^n$ be a closed semialgebraic cylinder and let 
$$f:\R\times W\ni (t,w)\rightarrow (f_1(t,w),\dots,f_m(t,w))\in \R^m$$ 
be a generically finite polynomial mapping. Assume that for every $i$ there is $\deg _tf_i\leq d$. Then the set $\s_f$ has degree of $\R$-uniruledness at most $d$ or it is empty.
\end{theorem}

\begin{corollary}\label{rwn1}
Let  $f=(f_1,...,f_m) :\R^n\rightarrow \R^m$ be a generically finite polynomial mapping with $d=\min_j\max_i\deg _{x_j}f_i$. Then the set $\s_f$ has degree of $\R$-uniruledness at most $d$ or it is empty.
\end{corollary}

Proofs of these facts are exactly the same as in the complex case (see section \ref{c}). To prove a real analog of Theorem \ref{multc} we need some new ideas. 






\begin{theorem}\label{multc1}
Let  $X$ be a closed semialgebraic set with degree of $\R$-uni\-ruled\-ness at most $d_1$, and let $f:X\rightarrow\R^m$ be a generically finite polynomial mapping of degree $d_2$. Then the set $\s_f$ has degree of $\R$-uniruledness at most $2d_1d_2$ or it is empty.
\end{theorem}

\begin{proof}
Let $a\in \s_f$ and let $(x_k)_{k\in\N}\subset X$ be a sequence of points such that $f(x_k)\to a$ and $x_k\to\infty$. 
By Proposition \ref{real} there exists a semialgebraic curve $W$ and a generically finite polynomial map 
$$\phi:\R\times W \ni (t,w)\to \phi(t,w)\in X,$$ 
such that $\deg_t\phi \leq d_1$, and there exists a sequence $(y_k)_{k\in\N}\subset\R\times W$
such that $f(\phi(y_k))\to a$ and $\phi(y_k)\to\infty$. In particular $a\in\s_{f\circ\phi}$. Let $\Gamma$ be a Zariski closure of $W$. We can assume that $\Gamma$ and its complexification $\Gamma^c$ are smooth. Denote $Z:=\R\times\Gamma$. We have the induced mapping $\phi: Z \to X$. Hence we have also the induced complex mapping $\phi^c: Z^c:=\C\times\Gamma^c\to X^c$, where $Z^c,X^c$ denote the complexification of $Z$ and $X$ respectively.

Let $\overline{\Gamma^c}$ be a smooth completion of $\Gamma^c$ and let us denote $\overline{\Gamma^c}\setminus\Gamma^c=\{ a_1,..., a_l\}.$ Now $\p^1\C\times\overline{\Gamma^c}$ is a projective closure of $Z^c$. The divisor 
$$D=\overline{Z^c}\setminus Z^c=\infty\times\overline{\Gamma^c}+\sum_{i=1}^l \p^1\C\times\{a_i\}$$ 
is a tree. The mapping $\phi$ induces a rational mapping $\phi: \overline{Z^c}\dashrightarrow\overline{X^c}$, where $\overline{X^c}$ denotes the
projective closure of $X^c.$ We can resolve points of indeterminacy of this mapping:

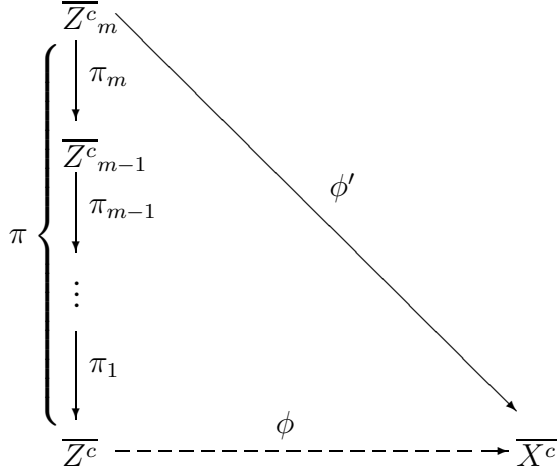
\vspace{3mm} 
\begin{center}\begin{picture}(240,160)(-40,40)
\put(-20,117.5){\makebox(0,0)[l]{$\pi\left\{\rule{0mm}{2.7cm}\right.$}}
\put(0,205){\makebox(0,0)[tl]{$\overline{Z^c}_m$}}
\put(0,153){\makebox(0,0)[tl]{$\overline{Z^c}_{m-1}$}}
\put(4,105){\makebox(0,0)[tl]{$\vdots$}}
\put(0,40){\makebox(0,0)[tl]{$\overline{Z^c}$}}
\put(170,40){\makebox(0,0)[tl]{$\overline{X^c}$}}
\put(80,50){\makebox(0,0)[tl]{$\phi$}}
\put(100,140){\makebox(0,0)[tl]{$\phi '$}}
\put(10,70){\makebox(0,0)[tl]{$\pi_1$}}
\put(10,130){\makebox(0,0)[tl]{$\pi_{m-1}$}}
\put(10,180){\makebox(0,0)[tl]{$\pi_m$}}
\put(5,190){\vector(0,-1){30}} \put(5,140){\vector(0,-1){30}}
\put(5,80){\vector(0,-1){33}}
\multiput(20,35)(8,0){17}{\line(1,0){5}}
\put(157,35){\vector(1,0){10}} \put(20,200){\vector(1,-1){150}}
\end{picture}
\end{center}

\vspace{3mm} 
Observe that $H:=\pi^{-1}(\overline{Z})$ has a structure of a real variety and $\R\times\Gamma\subset H$. Denote $Q:=\overline{Z_m}\cap \phi'^{-1}(X^c).$ Then the mapping $\phi': Q\to X^c$ is proper. Moreover, $Q=\overline{Z_m}\setminus
\phi'^{-1}(\overline{X^c}\setminus X^c).$ The divisor
$D_1=\phi'^{-1}(\overline{X^c}\setminus X^c)$ is connected  as a
complement of a semi-affine variety $\phi'^{-1}(X^c)$  (for
details see \cite{je99}, page $8$, Lemma~$4.5$). Note that  the divisor
$D'=\pi^*(D)$ is a tree. Hence the divisor $D_1\subset D'$ is also
a tree.

We can consider the mapping $f'=f\circ\phi': \R\times \Gamma\to \R^n$ as the regular mapping $f': Q\to \C^n.$ This mapping
induces a rational mapping from $H^c=\overline{Z_m^c}$ to $\p^n\C$. As
before we can resolve its points of indeterminacy:
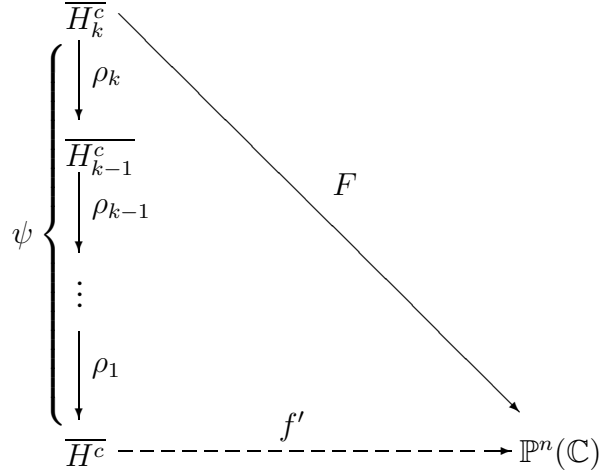
\vspace{3mm} 

\begin{center}\begin{picture}(240,160)(-40,40)
\put(-20,117.5){\makebox(0,0)[l]{$\psi\left\{\rule{0mm}{2.7cm}\right.$}}
\put(0,205){\makebox(0,0)[tl]{$\overline{H^c_k}$}}
\put(0,153){\makebox(0,0)[tl]{$\overline{H^c_{k-1}}$}}
\put(4,105){\makebox(0,0)[tl]{$\vdots$}}
\put(0,40){\makebox(0,0)[tl]{$\overline{H^c}$}}
\put(170,40){\makebox(0,0)[tl]{${\p^n(\C)}$}}
\put(80,50){\makebox(0,0)[tl]{$f'$}}
\put(100,140){\makebox(0,0)[tl]{$F$}}
\put(10,70){\makebox(0,0)[tl]{$\rho_1$}}
\put(10,130){\makebox(0,0)[tl]{$\rho_{k-1}$}}
\put(10,180){\makebox(0,0)[tl]{$\rho_k$}}
\put(5,190){\vector(0,-1){30}} \put(5,140){\vector(0,-1){30}}
\put(5,80){\vector(0,-1){33}}
\multiput(20,35)(8,0){17}{\line(1,0){5}}
\put(157,35){\vector(1,0){10}} \put(20,200){\vector(1,-1){150}}
\end{picture}
\end{center}

\vspace{3mm} 
Note that  the divisor $D_1'=\psi^*(D_1)$ is a tree.
Denote a proper transform of
$\infty\times\overline{\Gamma}$ by $\infty'\times\overline{\Gamma}$. 
It is an easy observation that
$F(\infty'\times\overline{\Gamma})\subset \pi_\infty$, where
$\pi_\infty$ denotes the hyperplane at infinity of $\p^n\C$.
Now $\s_{f'}= F(D_1\setminus F^{-1}(\pi_\infty)).$ The curve
$L=F^{-1}(\pi_\infty)$ is connected (the same argument as above). Now
by Proposition \ref{acykl} we have that every irreducible curve
$l\subset D_1$ (we have $l\cong \p^1\C$)
which does not belong to $L$ has at most one common point with
$L$. 

Let $R$ be an irreducible component of $\s_{f'}$. Hence $R$
is a curve. There is a curve $l\subset D_1,$ which has exactly one
common point with $L$, such that $R=F(l\setminus L).$ If $l$ is
given by blowing up of a real point, then $L$ also has a real
common point  with $l$ (since the conjugate point also is a
common point of $l$ and $L$). When we restrict to the real model
$l^r$ of $l$ we have  $l^r\setminus L\cong \R.$ Hence if we
restrict our consideration only to the real points and to the set
$Q^r:=\overline{\R\times W}\subset Q$ (we consider here a closure in the
euclidian topology) we see that the set $\s$ of
non-proper points of the mapping $f'|_{Q^r}$ is a union of
parametric curves $F(l^r\setminus L), \ l\in D_1,
\psi(l)\in \overline{Q^r}$, where the last closure is the closure in a real
projective space. Of course $a\in \s\subset\s_f$.
Similarly the set $\s_{f\circ\phi}$
 is a union of parametric curves $F(l^r\setminus L), \ l\in \psi^*(D'), \pi(\psi(l))\subset
\overline{\R\times W}\subset\overline{Z}$. Hence we can say that a "irreducible
component" of the set of non-proper points of $f'|_{Q^r}$ is also a
"irreducible" component of $\s_{f\circ\phi}$. Now we can finish the
proof by Theorem \ref{cxw} and the following Lemma \ref{xx}.
\end{proof}

\begin{lemma}\label{xx}
Let $\psi:\R\to\R^n$ be a parametric curve. If there exists a parametric curve $\phi: \R\to \R^n$ of degree at most $d$ with $\psi(\R)\subset \phi(\R)$, then $\psi(R)$ has degree of $\R$-uniruledness at most $2d$.
\end{lemma}

\begin{proof}
Indeed, let $\phi(t)=(\phi_1(t),\dots,\phi_n(t))$ and consider a
field 
$$\Le=\R(\phi_1,\dots,\phi_n).$$ 
By the L\"{u}roth Theorem \ref{luroth} there
exists a rational function $g(t)\in \R(t)$ such that $\Le=\R(g(t)).$
In particular $\phi_i(t)=f_i(g(t)).$ In fact we have two induced
mappings 
$$\overline{f} : \p^1\R\to \p^n\R\text{ and }
\overline{g}: \p^1\R\to \p^1\R.$$ 
Moreover,
$\overline{f}\circ \overline{g}=\overline{\phi}.$ Let $A_\infty$
denote the unique point at infinity of a Zariski closure of
$\psi(\R)$ and let $\infty=\overline{f}^{-1}(A_\infty).$ This implies
that $\#\overline{g}^{-1}(\infty)=1$ and we can assume that
$\overline{g}^{-1}(\infty)=\infty$, i.e., $g\in \R[t].$ Similarly
$f_i\in \R[t].$ Now if deg $g=1$ then $f:\R\to \R^n$ covers the whole
$\phi(R)$, because the image of $f$ is open and closed in
$\phi(R).$ Otherwise we can compose $f$ with suitable polynomial
of degree two to obtain the whole $\phi(R)$ in the image.

In the same way we can compose $f$ with suitable polynomial of degree one or two to obtain whole
$\psi(\R)$ in the image. In any case $\psi(\R)$ has a parametrization of degree bounded by
$2\deg f\leq 2d$.
\end{proof}

Theorem \ref{multc1} gives the following:

\begin{corollary}\label{multc3}
Let $X$ be a closed semialgebraic set which is $\R$-uniruled and let $f:X\rightarrow \R^m$ be a generically
finite polynomial mapping. Then every connected component of the set $\s_f$ is unbounded.
\end{corollary}

\section{A field of positive characteristic case}\label{positive}

In this section we assume $\K$ to be an arbitrary algebraically closed field. We begin with a useful lemma.

\begin{lemma}\label{tool}
Let $A\subset\K^n$ be an affine set, and let $f$ be a regular function on $\K^n$ not equal to
$0$ on any component of $A$. Suppose that for each $c\in \K^*$ the set 
$$A_c:=A\cap\{x\in\K^n:f(x)=c\}$$
has degree of $\K$-uniruledness at most $d$. Then $A_0$ also has degree of $\K$-uniruledness at most $d$.
\end{lemma}
\begin{proof}
Suppose that 
$$A=\{x\in\K^n:g_1(x)=0,\dots,g_r(x)=0\}.$$ 
For $a=(a_1,\dots,a_n)\in\K^n$ and $b=(b_{1,1}:\dots:b_{d,n})\in\p^{dn-1}$, let
$$\varphi_{a,b}:\K\ni t\rightarrow (a_1+b_{1,1}t+\dots+b_{1,d}^dt^d,\dots,a_n+b_{n,1}t+\dots+b_{n,d}^dt^d)\in\K^n$$ 
be a parametric curve of degree at most $d$. Let us consider a variety and a projection
$$\K^n\times\p^{dn-1}\supset V=\{(a,b)\in\K^n\times\p^{dn-1}:\forall_{t,i}\;g_i(\varphi_{a,b}(t))=0\;\text{and}$$
$$\forall_{t_1,t_2}\;f(\varphi_{a,b}(t_1))=f(\varphi_{a,b}(t_2))\}\ni(a,b)\rightarrow a\in\K^n.$$
The definition of the set $V$ says that parametric curves $\varphi_{a,b}$ are contained in $A$ and $f$ is constant on them. Hence $V$ is closed and the image of the projection is contained in $A$. Moreover the image contains all $A_c$ for $c\in \K^*$, since they are filled with
parametric curves of degree at most $d$. But since $\p^{dn-1}$ is complete and $V$ is closed we have that the image is closed, and hence it is the whole set $A$. In particular $A_0$ is contained in the image, so it is filled with parametric curves of degree at most $d$.
\end{proof}

We are ready to prove an analog of Theorem \ref{cn} for general fields.

\begin{theorem}\label{kn}
Let $\K$ be an arbitrary algebraically closed field, and let
$f:\K^n\rightarrow\K^m$ be a generically finite polynomial mapping of degree
$d$. Then the set $\s_f$ has degree of $\K$-uniruledness at most $d$ or it is empty.
\end{theorem}
\begin{proof}
If $n=1$ then the map is proper and $\s_f$ is empty. 

We consider the case $n\geq 2$. Due to Proposition \ref{graph} $$\s_f=p_2(\overline{\graph(f)}\setminus\graph(f)),$$
so it is enough to prove that the set
$$\overline{\graph(f)}\setminus\graph(f)$$ 
is filled with parametric curves of degree at most $d$. It is because we take images of curves under projection $p_2$ which is of degree $1$, 
and both sets are of the same dimension, so only on a codimension $1$ subvariety images of curves can become points. 

Denote the coordinates in $\p^n\times\K^m$ by $(x_{0} : \dots : x_n;x_{n+1},\dots,x_{n+m})$. Let us take an arbitrary point $$z\in\overline{\graph(f)}\setminus\graph(f)=\overline{\graph(f)}\cap\{x_0=0\}.$$ 
There exists $1\leq i\leq n$ such that $z_i\neq 0$. Consider sets 
$$A:=\overline{\graph(f)}\cap\{x_i\neq 0\}\text{ and }A_c:=\overline{\graph(f)}\cap\{x_i\neq 0\}\cap\{x_0=cx_i\}\text{ for }c\in\K.$$ 
The set $A$ is affine and sets $A_c$ satisfy
assumptions of Lemma \ref{tool} for $f=\frac{x_0}{x_i}$. Sets $A_c$ for $c\neq 0$ are filled with parametric
curves of degree at most $d$, because we can take $j\neq 0,i$ and curves
$$\K\ni t\rightarrow (c:a_1:\dots:a_{i-1}:1:\dots:a_{j-1}:t:a_{j+1}:\dots)\xrightarrow{f} A_c\subset\graph(f).$$
Hence the set 
$$A_0=\overline{\graph(f)}\cap\{x_i\neq 0\}\cap\{x_0=0\}$$ 
is also filled with such curves and we get that though $z$ passes a parametric curve of degree at most $d$ in $$\overline{\graph(f)}\cap\{x_0=0\},$$ 
this finishes the proof.
\end{proof}

By looking carefully at the proof of the above theorem we get the following slightly more general result:

\begin{theorem}\label{kw}
Let $\K$ be an arbitrary algebraically closed field, and let $f:\K\times W\rightarrow\K^m$ be a generically finite polynomial mapping of degree $d$. Then the set $\s_f$ is empty or all its components have degree of $\K$-uniruledness at most $d$, except possibly components of $p_2(\overline{\graph(f)}\cap \{(0:1:0:\dots)\}\times \K^m)$.
\end{theorem}

\begin{proof}
Suppose $\K\times W\subset\K\times\K^n$. As before it is enough to show that 
$$\overline{\graph(f)}\setminus\graph(f)$$ 
is filled with parametric curves of degree at most $d$. Let us take an arbitrary point $$z\in\overline{\graph(f)}\setminus\graph(f)=\overline{\graph(f)}\cap\{x_0=0\}.$$ 

If there exists $i\neq 1$ such that $z_i\neq 0$ then the idea from the proof of Theorem \ref{kn} works.
We apply Lemma \ref{tool} to the set 
$$A:=\overline{\graph(f)}\cap\{x_i\neq 0\}\text{ and }f=\frac{x_0}{x_i}.$$ 
Sets $A_c$ for $c\neq 0$ are filled with parametric
curves of degree at most $d$
$$\K\ni t\rightarrow (c:t:a_2:\dots:a_{i-1}:1:a_{i+1}\dots:a_n)\xrightarrow{f} A_c\subset\graph(f).$$

Otherwise $z=(0,1,0\dots,0)$. So for every component $C$ of $\s_f$ which is different from components of $p_2(\overline{\graph(f)}\cap \{(0:1:0:\dots)\}\times \K^m)$ its open dense subset is covered by parametric curves of degree at most $d$, thus by Proposition \ref{k-uniruledofdeg} component $C$ has degree of $\K$-uniruledness at most $d$.
\end{proof}

When applying Theorem \ref{kw} to two different directions we get:

\begin{corollary}\label{kkw}
Let $\K$ be an arbitrary algebraically closed field, and let $f:\K^2\times W\rightarrow\K^m$ be a generically finite polynomial mapping of degree $d$. Then the set $\s_f$ has degree of $\K$-uniruledness at most $d$ or it is empty.
\end{corollary}

\begin{proof}
Apply Theorem \ref{kw} to two directions. As previously it is enough to show that any component $C$ of  $$\overline{\graph(f)}\setminus\graph(f)$$  
has degree of $\K$-uniruledness at most $d$. From the first direction we get that for any component of $$\overline{\graph(f)}\setminus\graph(f)\setminus\overline{\graph(f)}\cap \{(0:1:0:0:\dots)\}\times \K^m,$$ 
has degree of $\K$-uniruledness at most $d$,
while from the second direction we get it for any component of 
$$\overline{\graph(f)}\setminus\graph(f)\setminus\overline{\graph(f)}\cap \{(0:0:1:0:\dots)\}\times \K^m.$$
Since $C$ is a component of at least one of them the assertion follows.
\end{proof}

\section{Remarks}

When we consider all generically finite mappings $\C^n\rightarrow\C^n$ of degree at most $D$, then by Theorem \ref{Sfkuni} hypersurfaces $\s_f$ are all $\C$-uniruled and by Theorem \ref{hiperc} their degree is bounded. It happens that they all are also of bounded degree of $\C$-uniruledness (namely by $D-1$, this follows from Theorem \ref{cn}). It is reasonable to ask the following question:

\begin{question}
Does for every $n$ and $d$ exist a universal constant $D=D(n,d)$, such that all $\K$-uniruled hypersurfaces in $\K^n$ of degree at 
most $d$ have degree of $\K$-uniruledness at most $D(n,d)$?
\end{question}

We can ask an even stronger question:

\begin{question}\label{22}
Does for every $n$ and $d$ there exist a universal constant $D=D(n,d)$, such that if a hypersurface in $\K^n$ of degree at most $d$ 
contains a parametric curve passing through $O=(0,\dots,0)$, then it also contains such a parametric curve of degree at most $D(n,d)$?
\end{question}

We can show the following equivalent condition to the affirmative answer to Question \ref{22}. Let us define
$$K_{n,d,D}=\{a=(a_1,\dots,a_s)\in\K^{\binom{n+d}{d}}:\text{ such that hypersurface in }\K^n\text{ defined}$$ $$\text{by }f_a=a_1+a_2x_1+\dots+a_sx_1^d\cdots x_n^d=0\text{ contains a parametric curve}$$ 
$$\text{passing through }O\text{ of degree at most }D\},$$
and let
$$K_{n,d}=\{(a_1,\dots,a_s)\in\K^{\binom{n+d}{d}}:\text{ such that hypersurface in }\K^n\text{ defined by}$$ $$f_a=a_1+a_2x_1+\dots+a_sx_1^d\cdots x_n^d=0\text{ contains a parametric curve}$$
$$\text{passing through }O\}.$$

\begin{proposition}
The set $K_{n,d,D}$ is closed.  
\end{proposition}

\begin{proof}
For $b=(b_{1,1}:\dots:b_{D,n})\in\p^M$, where $M=Dn-1$, let 
$$\varphi_{b}:\K\ni t\rightarrow (b_{1,1}t+\dots+b_{1,D}^Dt^D,\dots,b_{n,1}t+\dots+b_{n,D}^Dt^D)\in\K^n$$
be a parametric curve of degree at most $D$ passing through $O$. Consider a variety and a projection
$$\K^{\binom{n+d}{d}}\times\p^M\supset V=\{(a,b):\forall_{t}\;f_a(\varphi_{b}(t))=0\}\ni(a,b)\rightarrow a\in \K^{\binom{n+d}{d}}.$$
By the definition the set $V$ is closed and it consists of pairs $(a,b)$ corresponding to pairs $(f_a,\varphi_b)$ of a hypersurface and a parametric curve passing through $O$ contained in it. Hence the image equals to $K_{n,d,D}$. But since $\p^M$ is complete and $V$ is closed the image is also closed. 
\end{proof}

\begin{proposition}
Let $\K$ be an uncountable field. Then the following conditions are equivalent:
\begin{enumerate}
\item the set $K_{n,d}$ is closed,
\item the answer to Question \ref{22} is positive.  
\end{enumerate}
\end{proposition}

\begin{proof}
Obviously $K_{n,d}=\bigcup_{D\in\N}K_{n,d,D}$. 
Implication $(1)\Rightarrow(2)$ follows from Baire's Theorem \ref{baire}, the opposite one is trivial.
\end{proof}

Affirmative answers in the case of $\K=\C$ from section \ref{c}
and $\K=\R$ from section \ref{r} lead to the following natural conjecture:

\begin{conjecture}
Let $\K$ be an arbitrary algebraically closed field, and let
$f:\K^n\rightarrow \K^m$ be a generically finite polynomial mapping of degree
$d$. Then the set $\s_f$ has degree of $\K$-uniruledness at most $d-1$ or it is empty.
\end{conjecture}

Theorems \ref{multc} and \ref{multc1} suggests the following two conjectures:

\begin{conjecture}
Let $\K$ be an arbitrary algebraically closed field, $X$ be an affine variety with degree of $\K$-uniruledness at most $d_1$ and let $f:X\rightarrow\K^m$ be a generically finite polynomial mapping of degree $d_2$. Then the set $\s_f$ has degree of $\K$-uniruledness at most $d_1d_2$ or it is empty.
\end{conjecture}

\begin{conjecture}
Let $X$ be a closed semialgebraic set with degree of $\R$-uni\-ruled\-ness at most $d_1$, and let $f:X\rightarrow\R^m$ be a generically finite polynomial mapping of degree $d_2$. Then the set $\s_f$ has degree of $\R$-uniruledness at most $d_1d_2$ or it is empty.
\end{conjecture}

\pagebreak

\chapter{The set of fixed points of group actions}

In the whole Chapter, unless stated otherwise, $\K$ is assumed to be an arbitrary algebraically closed field.
We are going to remind the definition and a characterization of unipotent groups. The main goal of this Chapter is to show that under some conditions the set of fixed points $\Fix(\G)$ of action of an algebraic group $\G$ is $\K$-uniruled. 

In particular we prove:

\begin{theoremm}
Let $\G$ be a nontrivial connected unipotent algebraic group which acts effectively on an affine variety $X$. Then the set of fixed points of $\G$ is a $\K$-uniruled variety.
\end{theoremm}

\begin{theoremm}
Let $\G$ be an infinite connected algebraic group which acts effectively on $\K^n$, for $n\geq 2$. Assume that an irreducible hypersurface $H$ is contained in the set of fixed points of $\G$. Then $H$ is $\K$- uniruled.
\end{theoremm}

\section{Introduction}

We will use the following definition of unipotent groups.

\begin{definition}\label{series}
An algebraic group $\G$ is \textup{unipotent} if there exists a series of normal algebraic subgroups
$$0=\G_0\vartriangleleft \G_1\vartriangleleft\dots\vartriangleleft \G_r=\G,$$
such that $\G_i/\G_{i-1}\cong \G_a=(\K,+,0)$.
\end{definition}

Let $\G$ be a connected unipotent algebraic group, which acts effectively on a variety $X$. The set
of fixed points of this action was intensively studied (see \cite{bi, caso, gr, ho}). 
In particular Bia\l ynicki-Birula proved that if $X$ is an affine variety,
then $\G$ has no isolated fixed points.

Here we also consider the case when $X$ is an affine variety. We
generalize the result of Bia\l ynicki-Birula by proving, that
the set $\Fix(\G)$ of fixed points of $\G$ is in fact a $\K$-uniruled
variety. So through every point $x\in\Fix(\G)$ passes a
parametric curve.

We also show that if an arbitrary infinite connected algebraic group $\G$
acts effectively on $\K^n$ and the set of fixed points contains a
hypersurface $H$ then this hypersurface is $\K$-uniruled. It was known before for the case $\K=\C$ \cite{je03}.

\section{The set of fixed points of a unipotent group}

We begin with a quite natural observation:

\begin{lemma}\label{ginv}
Let algebraic group $\G$ act on affine varieties $X,Y$ and let $f:X\rightarrow Y$ be a generically finite $\G$-invariant mapping (for any $x\in X,g\in\G$ there is $gf(x)=f(gx)$). Then the set $\s_f$ of points at which map $f$ is not proper is also $\G$-invariant (for any $y\in \s_f,g\in\G$ there is $gy\in \s_f$). 
\end{lemma}

\begin{proof}
It is enough to show that the complement of the set $\s_f$ is
$\G$-invariant.

Suppose that $f$ is proper at $y\in Y$. This means that
there exists an open neighborhood $U$ of $y$ such that the mapping
$$f\vert_{f^{-1}(U)}:f^{-1}(U)\to U$$
is finite. We have the following diagram:

\begin{center}
\begin{picture}(240,160)(-40,40)
\put(140,160){\makebox(0,0)[tl]{$f^{-1}(gU)=gf^{-1}(U)$}}
\put(20,160){\makebox(0,0)[tl]{$f^{-1}(U)$}}
\put(30,40){\makebox(0,0)[tl]{$U$}}
\put(180,40){\makebox(0,0)[tl]{$gU$}}
\put(190,100){\makebox(0,0)[tl]{$f\vert_{f^{-1}(gU)}$}}
\put(40,100){\makebox(0,0)[tl]{$f\vert_{f^{-1}(U)}$}}
\put(95,50){\makebox(0,0)[tl]{$g$}}
\put(95,170){\makebox(0,0)[tl]{$g^{-1}$}}
\put(33,145){\vector(0,-1){100}} \put(45,35){\vector(1,0){125}}
\put(135,155){\vector(-1,0){75}} \put(183,145){\vector(0,-1){100}}
\end{picture}
\end{center}
\vspace{3mm}

Horizontal mappings are isomorphisms, hence by Corollary \ref{isoo} they are finite. The map $f\vert_{f^{-1}(gU)}$ is a composition of finite maps, so it is also finite. Hence $gy\notin \s_f$.
\end{proof}

The aim of this section is to prove the following:

\begin{theorem}\label{glowne}
Let $\G$ be a nontrivial connected unipotent algebraic group which acts effectively on an affine variety $X$. Then the set of fixed points of $\G$ is a $\K$-uniruled variety.
\end{theorem}

\begin{proof} By the induction on $\dim\G$ we can easily reduce the general case to the case of $\G=\G_a$. Indeed let 
$$0=\G_0\vartriangleleft \G_1\vartriangleleft\dots\vartriangleleft \G_r=\G$$
be a normal series like in Definition \ref{series}. Now $\Fix(\G)\subset\Fix(\G_{r-1})$. The set $\Fix(\G_{r-1})$ is invariant under $\G$ action, moreover $\G_{r-1}$ acts trivially on it. So the group $\G/\G_{r-1}\cong \G_a$ acts on $\Fix(\G_{r-1})$. If this action is trivial, then $\Fix(\G)=\Fix(\G_{r-1})$ and the assertion follows from the inductive assumption. Otherwise $\G_a$ acts effectively on $\Fix(\G_{r-1})$ so the set of fixed points $\Fix(\G)$ is $\K$-uniruled.

First assume that the field $\K$ is uncountable. Take a point $a\in\Fix(\G)$. By Propositions \ref{k-uniruled} and \ref{k-uniruledofdeg} it is enough to prove, that
there exists a parametric curve $S\subset\Fix(\G)$ passing through $a$.
Let $L$ be an irreducible curve in $X$ passing through $a$, which is
not contained in any orbit of $\G$ and which is not contained in
$\Fix(\G)$. 

Consider a surface $Y=L\times\G$. There is natural $\G$
action on $Y$. For $h\in\G$ and $y=(l,g)\in Y$ we put
$h(y)=(l,hg)\in Y$. Take a mapping
$$\Phi : L\times G\ni (x, g)\to gx\in X.$$ 
It is a generically finite polynomial mapping. 

Observe that it is $\G-$invariant,
$\Phi(gy)=g\Phi(y)$. Due to Lemma \ref{ginv} this implies that the set $\s_\Phi$ of points
at which the mapping $\Phi$ is not finite is also $\G$-invariant.

Theorem \ref{Sfsurface} gives that the set $\s_\Phi$ is $\K$-uniruled. Let $\s_\Phi=S_1\cup\dots\cup S_k$ be the decomposition of $\s_\Phi$ 
into parametric curves. Since the set $\s_\Phi$ is $\G$-invariant, we have that each curve $S_i$ is also
$\G$-invariant. 

Note that the point $a$ belongs to $\s_\Phi$,
because the fiber over $a$ has an infinite number of points. We can
assume that $a\in S_1$. We want to show that $S_1\subset\Fix(\G)$. Let $x\in S_1$, assume $x\notin\Fix(\G)$, then $\G x=S_1$ and $a$ would be in the orbit
of $x$, this is a contradiction with $a\in\Fix(\G)$. Hence $a\in S_1\subset\Fix(\G)$, which ends the proof of an uncountable field case.

Now assume that the field $\K$ is countable.
Let $\Le$ be an uncountable algebraically closed extension of $\K$. 
Then the group $\Le\G$ acts on $\Le X$ and $\Fix(\Le\G)=\Le\Fix(\G)$.
By the first part of our proof the variety
$\Fix(\Le\G)$ is $\Le$-uniruled, so due to Proposition \ref{indep} the set $\Fix(\G)$ is $\K$-uniruled.
\end{proof}

We get a following corollary:

\begin{corollary}[Bia\l ynicki-Birula \cite{bi}]
Let $\G$ be a nontrivial connected unipotent  group which acts
effectively on an affine variety $X$. Then $\G$ has no isolated fixed points.
\end{corollary}

\section{Hypersurface contained in the set\newline of fixed points}

We generalize the main result of \cite{je03} from $\C$ to an arbitrary algebraically closed field.

\begin{theorem}\label{glowne1}
Let $\G$ be an infinite connected algebraic group which acts
effectively on $\K^n$, for $n\geq 2$. Assume that an irreducible
hypersurface $H$ is contained in the set $\Fix(\G)$ of fixed points of $G$.
Then $H$ is $\K$-uniruled.
\end{theorem}

\begin{proof} Since $\G$ acts effectively on the affine space $\K^n$, then by the Chevalley Theorem (\cite{sh}, page $190$, Theorem C) we can assume that the group $\G$ is affine. 
In particular it contains either the subgroup $\G_m=(\K^*,\cdot,1)$ or the subgroup $\G_a=(\K,+,0)$ (see \cite{ii}). Thus we can assume that the group $\G$ is either $\G_m$ or $\G_a$.

As before we can assume that the field $\K$ is uncountable. Take a
point $a \in H$. By Propositions \ref{k-uniruled} and \ref{k-uniruledofdeg} it is enough to
prove, that there exists a parametric curve $S\subset H$ passing
through $a$. Let $L$ be a line in $\K^n$ going through $a$ such
that the set $L\cap\Fix(\G)$ is finite. Denote $L\cap H=\{a, a_1,\dots,a_m\}$. 

Consider a mapping
$$\phi:L\times\G\ni (x,g)\to gx\in\K^n.$$ 
Observe that $\phi(L\times\G)$ is a union of disjoint orbits of $\G$. This
implies that $\phi(L\times\G)\cap H=\{ a, a_1,\dots,a_m\}$. Take
$X=\overline{\phi(L\times\G)}$. Note that $X\cap H$ is a union of
curves. This means that there exists a curve $S\subset X\cap H$, which
contains the point $a$. However $S\subset \overline{X\setminus
\phi(L\times\G)}.$ This implies that $S\subset\s_\phi$ and we
conclude by Theorem \ref{Sfsurface}.
\end{proof}

\section{A real field case}

We give a real counterpart of Theorem \ref{glowne}.

\begin{theorem}\label{glowne2}
Let $\G$ be a real nontrivial connected unipotent algebraic group, which acts effectively and polynomially on a closed semialgebraic $\R$-uniruled set $X$. Then the set $\Fix(\G)$ of fixed points of $\G$ is also $\R$-uniruled. In particular it does not contain isolated points.
\end{theorem}

\begin{proof} By the induction on $\dim\G$ we
can easily reduce the problem to the case of $\G=\G_a$. 

Assume that $\G=\G_a$. Let $D$ be the degree of $\R$-uniruledness
of $X$. Take a point $a\in\Fix(\G)$. Let 
$$\phi:\G\times X\ni
(g,x)\to \phi(g,x)\in X$$ 
be a polynomial action of $\G$ on $X$.
This action also induces a polynomial action of the
complexification $\G^c=(\C,+,0)$ of $\G$ on $X^c$. We will denote
this action by $\overline{\phi}$. Assume that $\deg_g\phi\le d$.
By Definition \ref{r-uniruleddef} we have to prove, that
there exists a parametric curve $S\subset\Fix(\G)$ passing through
$a$ of degree bounded by max$(d,D)$. Let $L$ be a parametric
curve in $X$ passing through $a$. If it is contained in
$\Fix(\G)$, then the assertion is true. Otherwise consider a closed semialgebraic
surface $Y=L\times\G$. There is a natural $\G$ action on $Y$: 
$$\text{for }h\in \G\text{ and }y=(l,g)\in Y\text{ we put }h(y)=(l,hg)\in Y.$$ 
Take a generically finite polynomial mapping
$$\Phi:L\times\G\ni (x, g)\to \phi(g,x)\in X.$$ 

Observe that it is $\G$-invariant, which means
$\Phi(gy)=g\Phi(y)$. Lemma \ref{ginv} implies that the set $\s_\Phi$ of points
at which the mapping $\Phi$ is not finite is also $\G$-invariant.

Let $\s_\Phi=S_1\cup\dots\cup S_k$ be a decomposition of $\s_\Phi$ into irreducible components. Due to Theorem \ref{multc1} each $\s_i$ is a parametric curve. Since
the set $\s_\Phi$ is $\G-$invariant each $S_i$ is also $\G$-invariant. 

Note that the point $a$ belongs to $\s_\Phi$, because the fiber over $a$ has infinite number of points. We can assume that $a\in S_1$. Let us note that the point $a$ is also a fixed point for $\G^c$. We want to show that $S_1\subset\Fix(\G)$. Let $x\in S_1$. The set $S_1^c$ is also $\G^c$-invariant and if $x\not\in\Fix(\G)$ then $\G^cx=S_1^c$ and $a$ would be in the orbit of $x$, which is a contradiction. Hence $S_1\subset\Fix(\G)$ completes the proof.
\end{proof}

\begin{corollary}
Let $\G$ be a real nontrivial connected unipotent group which acts effectively and polynomially on a closed semialgebraic set $X$. If the set $\Fix(\G)$ of fixed points of $\G$ is nowhere dense in $X$, then it is $\R$-uniruled.
\end{corollary}

\begin{corollary}
Let $\G$ be a real nontrivial connected unipotent group which acts effectively and polynomially on a connected Nash submanifold $X\subset\R^n$. Then the set $\Fix(\G)$ is $\R$-uniruled.
\end{corollary}

\section{Remarks}

To finish we state:

\begin{conjecture} Let $\G$ be an algebraic group, which
acts effectively on $\K^n$. If $R$ is an irreducible component of
the set $\Fix(\G)$ of fixed points of $\G$, then $R$ is a $\K$-uniruled variety or it is a point.
\end{conjecture}
\pagebreak

\end{document}